\documentclass{article}
\usepackage{amsthm}
\newtheorem{theorem}{Theorem}
\newtheorem{proposition}{Proposition}
\newtheorem{lemma}{Lemma}
\newtheorem{corollary}{Corollary}
\usepackage{fullpage}

\usepackage{amssymb,amsmath}

\usepackage{hyperref}

\usepackage{graphics}
\usepackage{epstopdf}

\usepackage{algorithm}
\usepackage{algorithmic}
\usepackage{epstopdf}
\usepackage{graphicx}
\usepackage{multirow}
\usepackage{color}
\usepackage[colorinlistoftodos,bordercolor=orange,backgroundcolor=orange!20,linecolor=orange,textsize=scriptsize]{todonotes}

\usepackage{afterpage}

\usepackage{mathtools}
\DeclarePairedDelimiter\ceil{\lceil}{\rceil}

\newcommand{\norm}[1]{\left \lVert #1 \right\rVert}
\def\<#1,#2>{\langle #1,#2\rangle}

\newcommand{\eqdef}{\stackrel{\text{def}}{=}}
\newcommand{\R}{\mathbb{R}}

\newcommand{\bN}{\mathbb{N}}
\newcommand{\cX}{\mathcal{X}}

\newcommand{\dom}{\operatorname{dom}}
\newcommand{\dist}{\operatorname{dist}}
\newcommand{\prox}{\operatorname{prox}}

\newcommand{\tr}{\operatorname{trace}}

\newtheorem{ass}{Assumption}

\newtheorem{remark}{Remark}

\author{Olivier Fercoq\thanks{
	    T\'el\'ecom ParisTech, Universit\'e Paris-Saclay, Paris,  France,
              (\texttt{olivier.fercoq@telecom-paristech.fr})} \and Zheng Qu\thanks{Department of Mathematics, The University of Hong Kong,
Hong Kong, China,
              (\texttt{zhengqu@hku.hk})}
}

\newcommand{\TheTitle}{Adaptive restart of accelerated gradient methods under local quadratic growth condition}

\title{{\TheTitle}}

\begin{document}

\maketitle

\begin{abstract}
By analyzing accelerated proximal gradient methods under a local quadratic growth condition, we show that restarting these algorithms at any
frequency gives a globally linearly convergent algorithm. This result was previously
 known only for long enough frequencies. 

Then, as the rate of convergence depends on the match between the frequency and the quadratic error bound, we design
a scheme to automatically adapt the frequency of restart from
the observed decrease of the norm of the gradient mapping.
Our algorithm has a better theoretical bound than previously proposed
methods for the adaptation to the quadratic error bound of the objective.

We illustrate the efficiency of the algorithm on a Lasso problem and on a regularized logistic regression problem.
\end{abstract}

\section{Introduction}


\subsection{Motivation}

The proximal gradient method aims at minimizing composite
convex functions of the form $$F(x) = f(x) + \psi(x),\enspace x\in \R^n$$
where $f$ is differentiable with Lipschitz gradient 
and $\psi$ may be nonsmooth but has an easily computable proximal operator.  For a mild additional computational cost, accelerated gradient methods transform 
the proximal gradient method, for which the optimality gap 
$F(x_k) - F^\star$ decreases as $O(1/k)$, into an algorithm with 
``optimal'' $O(1/k^2)$ complexity~\cite{nesterov1983method}.  
Accelerated variants include the
dual accelerated proximal gradient~\cite{Nesterov05:smooth, nesterov2013gradient}, the accelerated proximal gradient method (APG)~\cite{tseng2008accelerated}
and FISTA~\cite{beck2009fista}. 
Gradient-type methods, 
also called first-order methods, 
are often used to solve large-scale problems  because of their good scalability and easiness
of implementation that facilitates parallel and distributed computations.


When solving a convex problem whose objective function satisfies a
local quadratic error bound (this is a generalization of strong convexity), classical
(non-accelerated) gradient and coordinate descent methods
automatically have a linear rate of convergence,
i.e. $F(x_k) - F^\star \in O((1-\mu)^k)$
for a problem dependent $0<\mu<1$~\cite{Necoara:Coupled,DrusvyatskLewis}, whereas one needs to know explicitly the strong convexity parameter in order
to set accelerated gradient and accelerated coordinate descent methods to have a linear rate of convergence, 
see for instance~\cite{lee2013efficient,UniversalCatalyst,lin2014accelerated,Nesterov:2010RCDM,nesterov2013gradient}. 
Setting the algorithm with an incorrect parameter may result in a slower algorithm, sometimes even slower than if
we had not tried to set an acceleration scheme~\cite{o2012adaptive}.
This is a major drawback of the method because in general, the strong convexity parameter is
difficult to estimate. 

In the context of accelerated gradient method with unknown strong convexity parameter, Nesterov~\cite{nesterov2013gradient} proposed a restarting scheme which adaptively approximates the strong convexity parameter. The same idea was exploited by Lin and Xiao~\cite{Lin2015} for sparse optimization.
Nesterov~\cite{nesterov2013gradient} also showed that,
instead of deriving a new method designed to work better for strongly convex functions,
one can restart the accelerated gradient method and get a linear convergence rate. However, the restarting frequency he proposed still depends explicitly on the strong convexity
of the function and so  O'Donoghue and Candes~\cite{o2012adaptive} introduced some heuristics to adaptively restart the 
algorithm and obtain good results in practice. 


\subsection{Contributions}
 
In this paper, we show that, if the objective function is convex and
satisfies a local quadratic error bound, we can restart accelerated  gradient methods
at {\em any} frequency and get a linearly convergent algorithm. 
The rate depends on an estimate of the quadratic error bound
and we show that for a wide range of this parameter,
one obtains a faster rate than without acceleration.
In particular, we do not require this estimate
to be smaller than the actual value.
In that way, our result supports and explains the practical success of arbitrary periodic restart 
for accelerated gradient methods.


Then, as the rate of convergence depends on the match between the frequency and the quadratic error bound, we design
a scheme to automatically adapt the frequency of restart from
the observed decrease of the norm of the gradient mapping.
The approach follows the lines of
\cite{nesterov2013gradient,Lin2015,LiuYang17}. We proved that, if our
current estimate of the local error bound were correct, the 
norm of the gradient mapping would decrease at a prescribed rate.
We just need to check this decrease and when the test fails, we have 
a certificate that the estimate was too large.

Our algorithm has a better theoretical bound than previously proposed
methods for the adaptation to the quadratic error bound of the objective.
In particular, we can make use of the fact that we know that the norm of
the gradient mapping will decrease even when we had a wrong estimate 
of the local error bound.

In Section~\ref{sec:problem} we recall
the main convergence results for accelerated gradient methods
and show that a fixed restart leads to a linear convergence rate.
In Section~\ref{sec:restart1}, we present our adaptive restarting rule. Finally, we present numerical experiments
on the lasso and logistic regression problem in Section~\ref{sec:expe}.


\section{Accelerated gradient schemes}
\label{sec:problem}

\subsection{Problem and assumptions}
We consider the following optimization problem:
\begin{equation}\label{eq-prob}
\begin{array}{ll}
 \displaystyle\min_{x\in \R^n} & F(x):=f(x)+\psi(x),
\end{array}
\end{equation}
where $f:\R^n\rightarrow \R$ is a differentiable convex function and $\psi:\R^n\rightarrow \R\cup\{+\infty\}$ is a proper, closed  and convex function. We denote by $F^\star$ the optimal value of~\eqref{eq-prob} and assume that the optimal solution set $\cX_\star$ is nonempty. Throughout the paper $\|\cdot\|$ denotes the Euclidean norm. 
For any positive vector $v\in \R^n_+$, we denote by $\|\cdot\|_v$ the weighted  Euclidean norm:
 \[\|x\|_v^2 \eqdef \sum_{i=1}^n v_i (x^i)^2,\]
and $\dist_v(x, \cX_\star)$ the distance of $x$ to the closed convex set $\cX_\star$ with respect to the norm
$\|\cdot\|_v$. 
  In addition we assume that $\psi$ is simple, in the sense that the proximal operator defined as
\begin{align}\label{a:proxoperator}
\prox_{v, \psi}(x):=\arg\min_{y} \left\{\frac{1}{2}\|x-y\|_v^2+\psi(y) \right\},
\end{align}
is easy to compute, for any positive vector $v\in \R_+^n$. We also make the following  \textit{smoothness} and \textit{local quadratic error bound} assumption.
\begin{ass}\label{ass1}
There is a positive vector $L\in \R_+^n$ such that
\begin{align}\label{a:smoothness}
f(x)\leq f(y)+\<\nabla f(y), x-y>+\frac{1}{2}\|x-y\|^2_L,\enspace \forall x,y \in \R^n.
\end{align}
\end{ass}

\begin{ass}\label{ass2}
For any $x_0\in \dom(F) $, there is $\mu_F(L,x_0)>0$ such that 
\begin{align}\label{a:strconv2}
F(x)\geq F^\star +\frac{\mu_F(L,x_0)}{2}\dist^2_L(x, \cX_\star), \enspace \forall x\in [F\leq F(x_0)],
\end{align}
where $[F\leq  F(x_0)]$ denotes the set of all $x$ such that $F(x)\leq F(x_0)$.
\end{ass}

Assumption~\ref{ass2} is also referred to as \textit{local quadratic growth condition}. It is known that Assumption~\ref{ass1} and~\ref{ass2} guarantee the linear convergence of proximal gradient method~\cite{Necoara:Coupled,DrusvyatskLewis} with complexity bound $O(\log(1/\epsilon)/\mu_F(L,x_0))$.

\subsection{Accelerated gradient schemes}

 We first recall in Algorithm~\ref{FISTA} and~\ref{APG}
two classical accelerated proximal gradient schemes. For identification purpose we refer to them respectively as FISTA (Fast Iterative Soft Thresholding Algorithm)~\cite{beck2009fista} and
APG (Accelerated Proximal Gradient) \cite{tseng2008accelerated}. As pointed out in~\cite{tseng2008accelerated}, the accelerated schemes were first proposed by Nesterov~\cite{NesterovBook}.

\begin{algorithm}
\begin{algorithmic}[1]
\STATE{Set $\theta_0 = 1$ and $z_0 = x_0$}.
\FOR{$k= 0,1,\cdots,K-1$} 
\STATE $y_k=(1-\theta_k)x_k+\theta_k z_k$ \\
\STATE
$x_{k+1} = \arg\min_{x\in \R^n} \big\{\< \nabla f(y_k), x-y_k>+\frac{1}{2} \|x-y_k\|^2_L+\psi(x) \big\}$ \\
\STATE $z_{k+1}=z_k+ \frac{1}{\theta_k}(x_{k+1}-y_k)$\\
\STATE $\theta_{k+1}=\frac{\sqrt{\theta_k^4 + 4 \theta_k^2} - \theta_k^2}{2}$
\ENDFOR
\STATE $\hat x\leftarrow x_{k+1}$
\end{algorithmic}
\caption{$\hat x\leftarrow$FISTA$(x_0,K)$~\cite{beck2009fista}}\label{FISTA}
\end{algorithm}
\begin{algorithm}
\begin{algorithmic}[1]
\STATE{Set $\theta_0 = 1$ and $z_0 = x_0$}.
\FOR{$k=0,1,\cdots,K-1$} 
\STATE $y_k=(1-\theta_k)x_k+\theta_k z_k$ \\
\STATE
$
z_{k+1}
=\arg\min_{z\in \R^n} \big\{\< \nabla f(y_k), z-y_k>+\frac{\theta_k}{2} \|z-z_k\|^2_L+\psi(z) \big\}
$ \\
\STATE $x_{k+1}=y_k+ \theta_k(z_{k+1}-z_k)$\\
\STATE $\theta_{k+1}=\frac{\sqrt{\theta_k^4 + 4 \theta_k^2} - \theta_k^2}{2}$
\ENDFOR
\STATE $\hat x\leftarrow x_{k+1}$
\end{algorithmic}
\caption{$\hat x\leftarrow$APG$(x_0,K)$~\cite{tseng2008accelerated}}
\label{APG}
\end{algorithm}

\begin{remark}
We have written the algorithms in a unified framework to emphasize their similarities.
Practical implementations usually consider only two variables: $(x_k, y_k)$ for FISTA and
$(y_k, z_k)$ for APG . 
\end{remark}

\begin{remark}
The vector $L$ used in the proximal operation step (line 4) of Algorithm~\ref{FISTA} and~\ref{APG} should satisfy Assumption~\ref{ass1}. If such $L$ is not known a priori, we can incorporate a line search procedure, see for example~\cite{Nesterov:2007composite}.
\end{remark}

The most simple restarted accelerated gradient method has a \textit{fixed restarting frequency}. This is Algorithm~\ref{alg:fixedrestart}, which restarts periodically Algorithm~\ref{FISTA} or Algorithm~\ref{APG}. Here, the restarting period is fixed to be some integer $K\geq 1$. In Section~\ref{sec:restart1} we will also consider \textit{adaptive restarting frequency} with a varying restarting period.
\begin{algorithm}
\begin{algorithmic}[1] 
\FOR{ $t=0,1,\cdots,$}
\STATE $x_{(t+1)K}\leftarrow \mbox{APG}(x_{tK}, K)$ or $x_{(t+1)K}\leftarrow \mbox{FISTA}(x_{tK}, K)$
\ENDFOR
\end{algorithmic}
\caption{FixedRES($x_0$, $K$)}\label{alg:fixedrestart}
\end{algorithm}

\subsection{Convergence results for accelerated gradients methods}

\subsubsection{Basic results}
\label{sec:convgrad}
In this section we gather a list of known results, shared by FISTA and APG, which will be used later to build restarted methods. Although all the results presented in this subsection have been proved or can be derived easily from existing results, for completeness all the proof  is given in Appendix. We first recall the following properties on the sequence $\{\theta_k\}$. 
\begin{lemma}\label{l:thetak}
The sequence $(\theta_k)$ defined by $\theta_0 = 1$ and $\theta_{k+1} = \frac{\sqrt{\theta_k^4 + 4 \theta_k^2} - \theta_k^2}{2}$ satisfies
\begin{align}&\label{athetabd}\frac{1}{k+1} \leq \theta_k \leq \frac{2}{k+2}
\\\label{arectheta}
&\frac{1-\theta_{k+1}}{\theta_{k+1}^2}=\frac{1}{\theta_k^2},\enspace \forall k=0,1,\dots
\\& \theta_{k+1}\leq \theta_k, \enspace \forall k =0,1,\dots, \label{atheradecr}
\end{align}
\end{lemma}
We shall also need the following relation of the sequences.
\begin{lemma}\label{l:xkzk}
The iterates of Algorithm~\ref{FISTA}  and Algorithm~\ref{APG} satisfy for all $k \geq 1$,
\begin{align}\label{a:xkzk}
x_{k+1}=(1-\theta_k)x_k+\theta_k z_{k+1}.
\end{align}
\end{lemma}

It is known that  the sequence of objective values $\{F(x_k)\}$ generated by accelerated schemes, in contrast to that generated by proximal gradient schemes, does not decrease monotonically. However, this sequence is always  upper bounded by the initial value $F(x_0)$. Following~\cite{Lin2015}, we refer to this result as the \textit{non-blowout} property of accelerated schemes.
\begin{proposition}\label{prop:nonblowout}The iterates of Algorithms~\ref{FISTA} and~\ref{APG} satisfy for all $k\geq 1$,
$$
F(x_k)\leq F(x_0).
$$
\end{proposition}
The non-blowout property of accelerated schemes can be found in many papers, see for example~\cite{Lin2015,WenChenPong}.  It will be repeatedly used in this paper to derive the linear convergence rate of restarted methods.
Finally recall the following fundamental property for accelerated schemes.
\begin{proposition}
\label{prop:fista_basic}
The iterates of Algorithms~\ref{FISTA}  and~\ref{APG} satisfy for all $k \geq 1$,
\begin{equation}
\label{eq:itcompl_fista}
\frac{1}{\theta_{k-1}^2}(F(x_{k}) - F^\star) + \frac{1}{2}\|z_{k}- x^\star\|^2_L \leq  \frac{1}{2}\|x_{0}-x^\star\|^2_L.
\end{equation}
where $x^\star$ is an arbitrary point in $\cX_\star$.
\end{proposition}
Again, Proposition~\ref{prop:fista_basic} is not new and can be easily derived from existing results, see Appendix for a proof. 
We derive from Proposition~\ref{prop:fista_basic} a direct corollary.
\begin{corollary}\label{coro:fista_basicdist}
The iterates of Algorithm~\ref{FISTA}  and Algorithm~\ref{APG} satisfy for all $k \geq 1$,
\begin{equation}
\label{eq:itcompl_fistadist}
\frac{1}{\theta_{k-1}^2}(F(x_{k}) - F^\star) + \frac{1}{2}\dist_L(z_{k},  \cX_\star)^2 \leq  \frac{1}{2}\dist_L(x_{0},\cX_\star)^2.
\end{equation}
\end{corollary}

\begin{remark}
All the results presented above hold without Assumption~\ref{ass2}.
\end{remark}

\subsubsection{Conditional linear convergence}
We first derive directly from Corollary~\ref{coro:fista_basicdist} and Assumption~\ref{ass2} the following decreasing property.
\begin{corollary}
\label{coro:fista_basic}
If $\hat x$ is the output of Algorithm~\ref{FISTA} or Algorithm~\ref{APG} with input $(x_0,K)$ and Assumptions~\ref{ass1} and~\ref{ass2} hold, then 
\begin{equation}
\label{eq:itcompl_fista_coro}
F(\hat x) - F^\star \leq  \frac{\theta^2_{K-1}}{\mu_F(L,x_0)} \left(F(x_0)-F^\star\right),
\end{equation}
and
\begin{equation}
\label{eq:itcompl_fista_coro_x}
\dist_L(\hat x,\cX_\star)^2 \leq  \frac{\theta^2_{K-1}}{\mu_F(L,x_0)} \dist_L(x_0,\cX_\star)^2.
\end{equation}
\end{corollary}
\begin{proof}
By Proposition~\ref{prop:nonblowout} and Assumption~\ref{ass2}, $\mu_F(L, \hat x)\geq \mu_F(L,x_0)$.  
Therefore,
$$
 \frac{1}{\theta_{K-1}^2}(F(\hat x) - F^\star) \overset{\eqref{eq:itcompl_fistadist}}{\leq}  \frac{1}{2}\dist_L(x_{0},\cX_\star)^2\leq \frac{1}{\mu_F(L,x_0)} \left(F(x_0)-F^\star\right),
$$
and
$$
\frac{\mu_F(L, x_0)}{2\theta_{K-1}^2} \dist_L(\hat x,\cX_\star)^2 \leq \frac{1}{\theta_{K-1}^2}(F(\hat x) - F^\star) \overset{\eqref{eq:itcompl_fistadist}}{\leq}  \frac{1}{2}\dist_L(x_{0},\cX_\star)^2.
$$
\end{proof}

\begin{proposition}\label{prop:fixedres}
Let $\{x_{tK}\}_{t\geq 0}$ be the sequence generated by Algorithm~\ref{alg:fixedrestart} with $K\geq 1$. If  Assumptions~\ref{ass1} and~\ref{ass2} hold, then we have
$$
F(x_{tK})-F^\star \leq \left(\frac{\theta^2_{K-1}}{\mu_F(L,x_0)} \right)^t \left(F(x_0)-F^\star\right),\enspace \forall t\geq 1,
$$
and
$$
\dist_L(x_{tK},\cX_\star)^2\leq \left(\frac{\theta^2_{K-1}}{\mu_F(L,x_0)} \right)^t \dist_L(x_0,\cX_\star)^2,\enspace \forall t\geq 1.
$$
\end{proposition}
\begin{proof}
From Corollary~\ref{coro:fista_basic},
we have
$$
F((t+1)K)-F^\star \leq \frac{\theta^2_{K-1}}{\mu_F(L,x_{tK})}  \left(F(tK)-F^\star\right),\enspace \forall t\geq 0,
$$
and
$$
\dist_L((t+1)K,\cX_\star)^2 \leq \frac{\theta^2_{K-1}}{\mu_F(L,x_{tK})}  \dist_L(tK,\cX_\star)^2,\enspace \forall t\geq 0.
$$
By the non-blowout property stated in Proposition~\ref{prop:nonblowout},
$$
F(x_{tK})\leq F(x_{(t-1)K})\leq \cdots F(x_0).
$$
Therefore, under Assumption~\ref{ass2},
$$
\mu_F(L,x_{tK})\geq \mu_F(L,x_0),\enspace \forall t\geq 0.
$$
It follows that
$$
F((t+1)K)-F^\star \leq \frac{\theta^2_{K-1}}{\mu_F(L,x_{0})}  \left(F(tK)-F^\star\right),\enspace \forall t\geq 0,
$$
and
$$
\dist_L((t+1)K,\cX_\star)^2 \leq \frac{\theta^2_{K-1}}{\mu_F(L,x_{0})}  \dist_L(tK,\cX_\star)^2,\enspace \forall t\geq 0.
$$
~
\end{proof}
\begin{remark}
Although Proposition~\ref{prop:fixedres}  provides a guarantee of  linear convergence when $\theta^2_{K-1}/\mu_F(L,x_0)<1$, it does not  give any information in the case when $K$ is too small to satisfy $\theta^2_{K-1}/\mu_F(L,x_0)<1$.
\end{remark}
For any $\mu>0$. Define:
\begin{align}\label{a:defK1}K(\mu):=\ceil*{ 2\sqrt{\frac{e}{\mu}} -1}.
\end{align}
By the right inequality of~\eqref{athetabd}, letting the restarting period $K\geq K(\mu_F(L,x_0))$ implies
\begin{align}\label{a:thetakmulessthane}
\frac{\theta^2_{K-1}}{\mu_F(L,x_0)}\leq e^{-1}. 
\end{align}
Therefore, if we know in advance $\mu_F(L,x_0)$ and restart Algorithm~\ref{FISTA} or~\ref{APG} every $K(\mu_F(L,x_0))$ iterations,
then after  computing
\begin{align}\label{a:fixedresoptimalcb}
K(\mu_F(L,x_0))\ln\frac{\dist_L(x_0,\cX_\star)^2}{\epsilon}=O\left(\sqrt{\mu^{-1}_F(L,x_0)}\ln{\epsilon^{-1}}\right)
\end{align}
number of proximal gradient mappings, we get a point $x$ such that $\dist_L(x,\cX_\star)^2\leq \epsilon$. 

\subsubsection{Unconditional linear convergence}
We now prove a contraction result on the distance to the optimal solution set. 
\begin{theorem}\label{th:linearconvergenceofFISTArestart}
If $\hat x$ is the output of Algorithm~\ref{FISTA} or Algorithm~\ref{APG} with input $(x_0,K)$ and Assumptions~\ref{ass1} and~\ref{ass2} hold, then 
\begin{equation}
\label{eq:itcompl_fista_quad_error_bound}
 \dist_L(\hat x, \cX_\star)^2 \leq \min\left( \frac{\theta_{K-1}^2}{\mu_F(L,x_0)}, \frac{1}{{ 1 + \frac{\mu_F(L, x_0)}{2\theta_{K-1}^2}} }\right)\dist_L(x_0, \cX_\star)^2 .
\end{equation}
\end{theorem}
\begin{proof}
For simplicity we denote $\mu_F = \mu_F(L, x_0)$.
 Note that by  Assumption~\ref{ass2} and Proposition~\ref{prop:nonblowout} we have for all $k\in \bN$,
\begin{align}\label{a:Fxkdistxk}
F(x_{k}) - F^\star \geq \frac{\mu_F}{2}\dist_L(x_{k},\cX_\star)^2.
\end{align}
For all $k\in \bN$ let us denote $x^\star_k$ the projection of $x_k$ onto $\cX_\star$.
For all $k \in \mathbb N$ and $\sigma_k \in [0,1]$, we have
\begin{align*}
&\frac 12 \norm{x_{k+1} - x^\star_{k+1}}^2_L = \frac{\sigma_k}{2} \norm{x_{k+1} - x^\star_{k+1}}^2_L + \frac{1-\sigma_k}{2} \norm{x_{k+1} - x^\star_{k+1}}^2_L \\
& \leq \frac{\sigma_k}{2} \norm{x_{k+1} - x^\star_{k+1}}^2_L + \frac{1-\sigma_k}{2} \norm{x_{k+1} - \theta_k x_0^\star - (1-\theta_k) x^\star_{k}}^2_L \\
& {\leq} \frac{\sigma_k}{\mu_F} (F(x_{k+1}) - F^\star) + \frac{1-\sigma_k}{2}\theta_{k}\norm{z_{k+1} - x^\star_0}^2_L + \frac{1-\sigma_k}{2}(1-\theta_{k})\norm{x_{k} - x^\star_k}^2_L,
\end{align*}
where the first inequality follows from the convexity of $\cX_\star$, and the second inequality follows from~\eqref{a:xkzk} and~\eqref{a:Fxkdistxk}.
Let us choose $\sigma_k = \frac{1}{1+\theta_k/\mu_F}$ so that 
\[
\frac{\sigma_k}{\mu_F} = (1-\sigma_k)\frac{\theta_k}{\theta_k^2}.
\]
We proceed as
\begin{align}
\label{eq:recurrence_on_deltak}
\frac 12 \dist_L(x_{k+1}, \cX_\star)^2 &\leq (1 - \sigma_k)\theta_k \left(\frac{1}{\theta_k^2} (F(x_{k+1}) - F^\star) + \frac{1}{2}\norm{z_{k+1} - x^\star_0}^2_L \right) \notag \\
& \qquad \qquad + \frac{(1-\sigma_k)(1-\theta_{k})}{2} \norm{x_{k} - x^\star_k}^2_L  \notag \\
& \overset{\eqref{eq:itcompl_fista}}{\leq} (1 - \sigma_k)\Big(\frac{\theta_k}{2} \norm{x_0-x^\star_0}^2_L + \frac{(1-\theta_{k})}{2} \norm{x_{k} - x^\star_k}^2_L\Big)  \notag \\
& = \frac{1}{1+\mu_F/\theta_k} \Big(\frac{\theta_k}{2} \dist_L(x_{0}, \cX_\star)^2 + \frac{(1-\theta_{k})}{2} \dist_L(x_k, \cX_\star)^2\Big)
\end{align}
Denote $\Delta_k = \dist_L(x_k, \cX_\star)^2$. 
Remark that 
\[\Delta_1 \leq \frac{1}{1+\mu_F/\theta_0} \Delta_0 \leq \frac{1}{1+0.5\mu_F/\theta_0^2} \Delta_0
\]
so that we may prove that $\Delta_k \leq  \frac{1}{1+0.5\mu_F/\theta_{k-1}^2} \Delta_0$ by induction.
Let us assume that $\Delta_k \leq \frac{1}{1 +0.5 \mu_F/\theta_{k-1}^2} \Delta_0$. Then using \eqref{eq:recurrence_on_deltak}
\begin{align*}
\Delta_{k+1} &\leq \frac{1}{1+\mu_F/\theta_k} \Big(\theta_k \Delta_0 + (1-\theta_{k}) \Delta_k\Big) \\
& \leq \frac{1}{1+\mu_F/\theta_k} \Big(\theta_k \Delta_0 + \frac{1-\theta_{k}}{1 + 0.5\mu_F/\theta_{k-1}^2} \Delta_0\Big) \\
& = \frac{\theta_k(1+0.5\mu_F/\theta_{k-1}^2) + (1-\theta_k)}{(1+\mu_F/\theta_k)(1+0.5\mu_F/\theta_{k-1}^2)}\Delta_0
\end{align*}
Using~\eqref{arectheta} one can then easily check that
\begin{multline*}
\frac{\theta_k(1+0.5\mu_F/\theta_{k-1}^2) + (1-\theta_k)}{(1+\mu_F/\theta_k)(1+0.5\mu_F/\theta_{k-1}^2)} \leq \frac{1}{1+0.5\mu_F/\theta_{k}^2} \Leftrightarrow 0\leq 2\theta_k^3+\mu_F(1-\theta_k)
\end{multline*}
and so Inequality~\eqref{eq:itcompl_fista_quad_error_bound} comes by combining this with Proposition~\ref{prop:fixedres}. 
\end{proof}

Theorem~\ref{th:linearconvergenceofFISTArestart} allows us to derive immediately an explicit linear convergence rate of  Algorithm~\ref{alg:fixedrestart}, regardless of  the choice of $K\geq 1$.
\begin{corollary}\label{coro:inearconvergenceofFISTArestart}
Let $\{x_{tK}\}_{t\geq 0}$ be the sequence generated by Algorithm~\ref{alg:fixedrestart} with fixed restarting period $K\geq 1$.  If Assumptions~\ref{ass1} and~\ref{ass2} hold, then
$$
 \dist_L(x_{tK} , \cX_\star)^2 \leq \left( \min\left(\frac{\theta^2_{K-1}}{\mu_F(L,x_0)},  \frac{1}{1 + \frac{\mu_F(L, x_0)}{2\theta_{K-1}^2}} \right)\right)^{{t}}\dist_L(x_0 ,\cX_\star)^2.
$$
\end{corollary}
By Corollary~\ref{coro:inearconvergenceofFISTArestart}, the number of proximal mappings needed to reach an $\epsilon$-accuracy on the distance is bounded by
$$
\frac{2K}{\log(1+\frac{\mu_F(L,x_0)}{2\theta^2_{K-1}})}\log\frac{\dist_L(x_0,\cX_\star)}{\epsilon}.
$$
In particular, if we choose $K\sim 1/\sqrt{\mu_F}$, then we get an iteration complexity bound 
\begin{align}\label{a:optimalbound}
O(1/\sqrt{\mu_F} \log(1/\epsilon)).
\end{align}

\begin{remark}
In~\cite{WenChenPong}, the local linear convergence of the sequence generated by FISTA with arbitrary (fixed or adaptive) restarting frequency was proved. Our Theorem~\ref{th:linearconvergenceofFISTArestart} not only yields the global linear convergence of such sequence, but also gives an explicit bound on the convergence rate. Also, note that although an asymptotic linear convergence rate can be derived from the proof of Lemma 3.6 in~\cite{WenChenPong}, it can be checked that the asymptotic rate in~\cite{WenChenPong}  is not as good as ours. In fact, an easy calculation shows that even restarting with optimal period $K\sim 1/\sqrt{\mu_F}$, their asymptotic rate only leads to the complexity bound  $O(1/\mu^2_F\log(1/\epsilon))$. Moreover, our restarting scheme is more flexible, because the internal block in Algorithm~\ref{alg:fixedrestart} can be replaced by any scheme which satisfies all the properties presented in Section~\ref{sec:convgrad}.
\end{remark}

\section{Adaptive restarting of accelerated gradient schemes}
\label{sec:restart1}

Although Theorem~\ref{th:linearconvergenceofFISTArestart} guarantees a linear convergence of the restarted method (Algorithm~\ref{alg:fixedrestart}), it requires the knowledge of $\mu_F(L,x_0)$ to attain the complexity bound~\eqref{a:optimalbound}. 
In this section, we show how to
combine Corollary~\ref{coro:inearconvergenceofFISTArestart} with Nesterov's adaptive restart method, first proposed in~\cite{Nesterov:2007composite},  in order to
obtain a complexity bound close to~\eqref{a:optimalbound} that does not depend on a guess on $\mu_F(L,x_0)$.

\subsection{Bounds on gradient mapping norm}

We first show the following inequalities that generalize similar ones in~\cite{Nesterov:2007composite}. Hereinafter we define the proximal mapping:
$$T(x):= \arg\min_{y\in \R^n} \left\{\< \nabla f(x), y-x>+\frac{1}{2} \|y-x\|^2_L+\psi(y) \right\}.$$      
\begin{proposition}
\label{prop:estimatedistXstarbb1}  If Assumption~\ref{ass1} and~\ref{ass2} hold, then for any $x\in \R^n$, we have
\begin{align}\label{a:TLyleqF1}
& \norm{T(x) - x}^2_L \leq  {2}\left(F(x) - F(T(x)) \right) \leq 2\left(F(x) - F^\star\right),
\end{align}
\begin{align}\label{a:normTLdistx}
\dist_L(T(x),\cX_\star) \leq \frac{4}{\mu_F(L,x)}  \norm{x - T(x) }_L ,
\end{align}
and
\begin{align}\label{a:FTLy1}
& F(T(x))-F^\star \leq  \frac{8\norm{x - T(x) }_L^2}{\mu_F(L,x)}.
\end{align}
\end{proposition}
\begin{proof}
The inequality~\eqref{a:TLyleqF1} follows directly from~\cite[Theorem 1]{Nesterov:2007composite}.
By the convexity of $F$,  for any $q \in \partial F(T(x)) $ and $ x^\star \in \cX_\star$,
\[
 \langle q, T(x) - x^\star\rangle \geq F(T(x)) - F^\star, \enspace \forall q \in \partial F(T(x)), x^\star \in \cX_\star.\]
Furthermore, by \cite[Theorem 1]{Nesterov:2007composite}, for any $q \in \partial F(T(x)) $ and $ x^\star \in \cX_\star$,
\begin{align*}
2\norm{x - T(x) }_L \norm{T(x)  - x^\star}_L {\geq} \langle q, T(x) - x^\star\rangle.
\end{align*}
Therefore,
\begin{align}\label{a:yTLyFTLy1}
2  \norm{x - T(x) }_L \dist_L(T(x),\cX_\star) {\geq}  F(T(x)) - F^\star .
\end{align}
By~\eqref{a:TLyleqF1}, we know that $F(T(x))\leq F(x)$ and in view of~\eqref{a:strconv2},
$$
 F(T(x)) - F^\star \geq \frac{\mu_F(L,x)}{2}\dist^2_L(T(x), \cX_\star).
$$
Combining the latter two inequalities, we get:
\begin{align}\label{a:distqv1}
2 \norm{x - T(x) }_L  {\geq}  \frac{\mu_F(L,x)}{2} \dist_L(T(x),\cX_\star) .
\end{align}
Plugging~\eqref{a:distqv1} back to~\eqref{a:yTLyFTLy1} we get~\eqref{a:FTLy1}.
\end{proof}
\begin{remark}
Condition~\eqref{a:normTLdistx} is usually referred to as an \textit{error bound condition}.
It was proved in~\cite{DrusvyatskLewis} that under Assumption~\ref{ass1}, the error bound condition is equivalent to the quadratic growth condition~\eqref{a:strconv2}. 
\end{remark}

\begin{corollary}
Suppose Assumptions~\ref{ass1} and~\ref{ass2} hold and
denote $\mu_F=\mu_F(L,x_0)$.
The iterates of Algorithm~\ref{alg:fixedrestart} satisfy for all $t\geq 1$,
\begin{align}\label{a:fixedine1}
\|T(x_{tK})-x_{tK} \|_L^2 \leq  {\theta^2_{K-1}} \left( \min\left(\frac{\theta^2_{K-1}}{\mu_F},  \frac{1}{1 + \frac{\mu_F}{2\theta_{K-1}^2}} \right)\right)^{{t-1}} \dist_L(x_0, \cX_\star)^2.
\end{align}
Moreover, if  there is $u_0\in \R^n$ such that $x_0=T(u_0)$ and denote $\mu_F'=\mu_F(L,u_0)$, then
\begin{align}\label{a:fixedine2}
\|T(x_{tK})-x_{tK} \|_L^2 \leq \frac{16}{\mu_F'}
\left(\frac{\theta^2_{K-1}}{\mu_F'}\right) \left( \min\left(\frac{\theta^2_{K-1}}{\mu_F'},  \frac{1}{1 + \frac{\mu_F'}{2\theta_{K-1}^2}} \right)\right)^{{t-1}}\|x_0-u_0 \|_L^2.
\end{align}
\end{corollary}
\begin{proof}
By~\eqref{a:TLyleqF1} and~\eqref{eq:itcompl_fistadist} we have
$$
\|T(x_{tK})-x_{tK}\|^2_L\leq 2(F(x_{tK})-F^\star)\leq \theta^2_{K-1}\dist_L(x_{(t-1)K},\cX_\star)^2.
$$
Now applying  Corollary~\ref{coro:inearconvergenceofFISTArestart} we get~\eqref{a:fixedine1}.
If in addition $x_0=T(u_0)$, then in view of~\eqref{a:normTLdistx},
$$
\dist_L(x_{0},\cX_\star)^2 \leq \left(\frac{4}{\mu_F'}\right)^2 \|x_0-u_0 \|_L^2.
$$
The second inequality~\eqref{a:fixedine2}  then follows from combining the latter inequality, \eqref{a:fixedine1} and the fact that $\mu_F' \leq \mu_F$.
\end{proof}

\subsection{Adaptively restarted algorithm}

Inequality~\eqref{a:fixedine2} provides a way to test whether our guess on $\mu_F(L,x_0)$ is too large. Indeed, let $x_0=T(u_0)$. Then if $\mu \leq \mu_F(L,u_0)\leq \mu_F(L,x_0)$ and we run Algorithm~\ref{alg:fixedrestart} with $K=K(\mu)$, then necessarily we have
\begin{align}\label{a:fixedine3}
\|T(x_{tK})-x_{tK} \|_L^2 \leq \frac{16}{\mu}
\left(\frac{\theta^2_{K-1}}{\mu}\right) \left( \min\left(\frac{\theta^2_{K-1}}{\mu},  \frac{1}{1 + \frac{\mu}{2\theta_{K-1}^2}} \right)\right)^{{t-1}}\|x_0-u_0 \|_L^2.
\end{align}
It is essential that both sides of~\eqref{a:fixedine3} are computable, so that we can check this inequality for each estimate $\mu$. If~\eqref{a:fixedine3} does not hold, then we know that $\mu>\mu_F(L,u_0)$. The idea was originally proposed by Nesterov in~\cite{Nesterov:2007composite} and later generalized in~\cite{Lin2015}, where instead of restarting, they incorporate the estimate $\mu$ into the update of $\theta_k$. As a result, the complexity analysis only works for strongly convex objective function and seems not to hold under Assumption~\ref{ass2}.

Our adatively restarted algorithm is described in Algorithm~\ref{alg:approxresdeter12}. We start from an initial estimate $\mu_0$, and restart Algorithm~\ref{FISTA} or~\ref{APG} with period $K(\mu_s)$ defined in~\eqref{a:defK1}. Note that  by~\eqref{a:thetakmulessthane},
$$
 \min\left(\frac{\theta^2_{K_s-1}}{\mu_s},  \frac{1}{1 + \frac{\mu_s}{2\theta_{K_s-1}^2}}\right)=\frac{\theta^2_{K_s-1}}{\mu_s}.
$$ At the end of each restarting period, we test condition~\eqref{a:fixedine3}, the opposite of which is given by the first inequality at Line~\ref{line:compute_u}.
 If it holds then we continue with the same estimate $\mu_s$ and thus the same restarting period, otherwise we decrease $\mu_s$ by one half and repeat. Our stopping criteria is based on the norm of proximal gradient,  same as in related work~\cite{Lin2015, LiuYang17}. 
\begin{algorithm}
\begin{algorithmic}[1]
\STATE{\textbf{Parameters:}  $\epsilon$, $\mu_0$}
\STATE   $s\leftarrow -1$, $t_s\leftarrow 0$
\STATE $x_{s,t_s}\leftarrow x_{0}$
\STATE $x_{s+1,0}\leftarrow T(x_{s,t_s})$
\REPEAT
\STATE $s\leftarrow s+1$
\STATE $C_{s}\leftarrow \frac{ 16 \|x_{s,0}-x_{s-1,t_{s-1}}\|_L^2} {\mu_s}$ \label{line:defineCs}
\STATE $K_s\leftarrow K(\mu_s)$
\STATE $t\leftarrow 0$
\REPEAT \label{line:start}
\STATE $x_{s, t+1}\leftarrow \mbox{Algorithm}~\ref{FISTA}(x_{s,t}, K_s)$ or $x_{s, t+1}\leftarrow \mbox{Algorithm}~\ref{APG}(x_{s, t}, K_s)$   \label{line:fistaorapg}
\STATE $t\leftarrow t+1$
\UNTIL{$\|T(x_{s,t})-x_{s,t}\|_L^2 > { C_s(\theta_{K_s-1}^2/\mu_s)^t}$ or $\|T(x_{s,t})-x_{s, t}\|_L^2 \leq \epsilon $\label{line:compute_u}}
\STATE $t_s\leftarrow t$
\STATE $x_{s+1,0}\leftarrow T(x_{s,t_s})$
\STATE $\mu_{s+1}\leftarrow \mu_s/2$
\UNTIL{$ \|x_{s+1,0}-x_{s, t_s}\|_L^2 \leq \epsilon$ \label{line:stop_test}}
\STATE $\hat x\leftarrow T(x_{s,t_s})$, $ \hat s \leftarrow s$, $\hat N\leftarrow 1+\displaystyle\sum_{s=0}^{\hat s} (t_sK_s+1)$
\end{algorithmic}
\caption{$(\hat x,  \hat s,\hat N)\leftarrow$AdaRES($x_{0}$)}\label{alg:approxresdeter12}
\end{algorithm}

\begin{remark}
Although  Line \ref{line:compute_u}
of Algorithm \ref{alg:approxresdeter12} requires to compute the proximal gradient mapping $T(x_{s,t})$, one should remark that 
this $T(x_{s,t})$ is in fact given by the first iteration of Algorithm~\ref{FISTA}($x_{s,t},K_s$) or Algorithm~\ref{APG}($x_{s,t},K_s$). Hence, except for $t=t_s$, the computation of $T(x_{s,t})$ does not incur additional computational cost. Therefore, the output $\hat N$ of Algorithm~\ref{alg:approxresdeter12} records the total number of proximal gradient mappings needed to get $\|T(x)-x\|^2_L\leq \epsilon$.
\end{remark}

We first show the following non-blowout property for Algorithm~\ref{alg:approxresdeter12}.
\begin{lemma}\label{l:nonblowoutalgo}
For any $-1\leq s\leq \hat s$ and $0\leq t\leq t_s$ we have
$$
F(T(x_{s,t}))\leq F(x_{s,t})\leq  F(x_0).
$$
\end{lemma}
\begin{proof}
Since $x_{s,t+1}$ is the output of Algorithm~\ref{FISTA} or~\ref{APG} with input $x_{s,t}$, we know  by Proposition~\ref{prop:nonblowout} that
$$
F(x_{s,t_s})\leq F(x_{s,t_s-1}) \leq \ldots \leq F(x_{s,0}).
$$
By~\eqref{a:TLyleqF1},
$$
F(x_{s+1,0})=F(T(x_{s,t_s}))\leq F(x_{s,t_s}) \leq F(x_{s,0}).
$$
The right inequality then follows by induction since $x_{-1,0}=x_0$. The left inequality follows from~\eqref{a:TLyleqF1}.
\end{proof}

\begin{lemma}\label{l:FboundTk}
For any $0\leq s\leq \hat s$,
if $\mu_s\leq \mu_F(L,x_0)$, then
\begin{align}\label{a:FxFu0}
F(x_{s,t})-F^\star &\leq \left(\frac{\theta^2_{K_s-1}}{\mu_s} \right)^t(F(x_{s,0})-F^\star),\enspace \forall 1\leq t\leq t_s.
\end{align}
\end{lemma}
\begin{proof}
This is a direct application of Proposition~\ref{prop:fixedres} and Lemma~\ref{l:nonblowoutalgo}.
\end{proof}
From Lemma~\ref{l:nonblowoutalgo} and~\ref{l:FboundTk}, we obtain immediately the following key results for the complexity bound of Algorithm~\ref{alg:approxresdeter12}.
\begin{corollary}\label{coro:uminusxCk}
Let $C_s$ be the constant defined in Line~\ref{line:defineCs} of Algorithm~\ref{alg:approxresdeter12}. We have 
\begin{align}\label{a:uminusxCk1}
C_s\leq \frac{32(F(x_0)-F^\star)}{\mu_s},\enspace \forall s\geq 0.
\end{align}
 If for any $0\leq s\leq \hat s$ we have $\mu_s\leq \mu_F(L,x_0)$, then
 \begin{align}\label{a:uminusxCk3}
 \|T(x_{s,t})-x_{s,t} \|_L^2 \leq C_s \left(\frac{\theta^2_{K_s-1}}{\mu_s}\right)^t,\enspace \forall 0\leq t\leq t_s,
 \end{align}
 and
\begin{align}\label{a:uminusxCk2}
\|T(x_{s,t})-x_{s,t} \|_L^2\leq 2e^{-t}(F(x_0)-F^\star),\enspace \forall 0\leq t\leq t_s.
\end{align}
\end{corollary}
\begin{proof}
The bound on $C_s$ follows from~\eqref{a:TLyleqF1} applied at $x_{s-1,t_{s-1}}$ and Lemma~\ref{l:nonblowoutalgo}. The second bound can be derived from~\eqref{a:TLyleqF1},~\eqref{a:FTLy1} and  Lemma~\ref{l:FboundTk}. For the third one, it suffices to apply Lemma~\ref{l:nonblowoutalgo}, Lemma~\ref{l:FboundTk} together with~\eqref{a:TLyleqF1} and the fact that 
\begin{align}\label{a:thetamue}
\frac{\theta^2_{K_s-1}}{\mu_s}\leq e^{-1},\enspace \forall s\geq 0.
\end{align}
\end{proof}

\begin{proposition}\label{th:adapresdeter}
Consider Algorithm~\ref{alg:approxresdeter12}. If for any $0\leq s\leq \hat s$ we have
$\mu_s\leq \mu_F(L,x_0)$, then  $\hat s=s$ and
\begin{align}\label{a:adapresmuoleq}
t_s\leq \ceil*{\ln \frac{2(F(x_0)-F^\star)}{\epsilon}}
.\end{align}
\end{proposition}
\begin{proof}
If $\mu_s\leq \mu_F(L,x_0)$, by~\eqref{a:uminusxCk3}, when the inner loop terminates we necessarily have $$
\|T(x_{s,t})-x_{s,t}\|_L^2\leq \epsilon.
$$
Therefore $\hat s=s$. Then, \eqref{a:adapresmuoleq} is derived from~\eqref{a:uminusxCk2}.
\end{proof}

\begin{theorem}
\label{th:adapresdeter2}
Suppose Assumptions~\ref{ass1} and~\ref{ass2} hold.
Consider the adaptively restar\-ted accelerated gradient  Algorithm~\ref{alg:approxresdeter12}. 
If the initial estimate of the local quadratic error bound satisfies   $\mu_0\leq \mu_F(L,x_0)$ then the number of iterations $\hat N$ is bounded by 
\begin{align}\label{a:adapresmuoleq33}
\hat N\leq  \ceil*{ {2}\sqrt{\frac{e}{\mu_0}}-1}
\ceil*{\ln \frac{2(F(x_0)-F^\star)}{\epsilon}}+2
.\end{align}
If $\mu_0>\mu_F(L,x_0)$, then 
\begin{align}
\label{a:adapresmuogeq33}
\hat N \leq &1+\ceil*{\log_2 \frac{\mu_0}{\mu_F(L,x_0)}}+{\frac{2\sqrt{e}}{\sqrt{2}-1}\left(\sqrt{\frac{2}{{\mu_F(L,x_0)}}}-\sqrt{\frac{1}{\mu_0}}\right)}\ceil*{\ln \frac{32(F(x_0)-F^\star)}{\epsilon\mu_F(L,x_0)}}\notag \\&\qquad +\ceil*{ {\frac{2\sqrt{2e}}{\sqrt{\mu_F(L,x_0)}}} }\ceil*{\ln \frac{2(F(x_0)-F^\star)}{\epsilon}} .
\end{align}
\end{theorem}
\begin{proof}
The first case is a direct application of Proposition~\ref{th:adapresdeter}.

Let us now concentrate on the case $\mu_0 > \mu_F(L, x_0)$.
For simplicity we denote $\mu_F=\mu_F(L,x_0)$.
Define $
\bar l  = \ceil*{\log_2\mu_0-\log_2\mu_F} \geq 1.
$ Then $\mu_{\bar l}\leq \mu_F$ and by Proposition~\ref{th:adapresdeter}, we know that $\hat s\leq \bar l$ and if $\hat s=\bar l$,
\begin{align}\label{a:tbarlbound}
t_{\bar l}\leq \ceil*{\ln \frac{2(F(x_0)-F^\star)}{\epsilon}}.
\end{align}
Now we consider $t_s$ for $0\leq s\leq \bar l-1$. Note that 
$\epsilon<\|T(x_{s,t})-x_{s,t}\|^2_L\leq  C_s(\theta^2_{K_s-1}/\mu_s)^t$ cannot hold for $t$  satisfying
\begin{align}\label{a:CsThe}
C_s(\theta^2_{K_s-1}/\mu_s)^t\leq \epsilon.
\end{align}
By~\eqref{a:uminusxCk1}, 
\begin{align}\label{a:Clupperbound}
C_s  \leq  \frac{32(F(x_0)-F^\star)}{\mu_F}, \enspace  0\leq s \leq \bar l-1.
\end{align}
In view of~\eqref{a:thetamue} and~\eqref{a:Clupperbound}, ~\eqref{a:CsThe} holds for any $t\geq 0$ such that 
$$ 32(F(x_0)-F^\star)e^{-t}\leq \epsilon \mu_F. $$ Therefore, 
\begin{align}\label{atlbound}
t_{s}\leq \ceil*{\ln \frac{32(F(x_0)-F^\star)}{\epsilon\mu_F}},\enspace  0\leq s \leq \bar l-1.
\end{align}
By the definition~\eqref{a:defK1},
\begin{align*}
&K_0+\cdots+K_{\bar l-1}=
\sum_{s=0}^{\bar l -1} \ceil*{ 2\sqrt{\frac{e}{\mu_s}} -1} \leq   \sum_{s=0}^{\bar l -1} 2\sqrt{\frac{2^{s}e}{\mu_0}} 
\\& = {\frac{2\sqrt{e/\mu_0}}{\sqrt{2}-1} \left(\sqrt{2^{{\bar l}}}-1\right)}
\leq {\frac{2\sqrt{e}}{\sqrt{2}-1}\left(\sqrt{\frac{2}{{\mu_F}}}-\sqrt{\frac{1}{\mu_0}}\right)}
\end{align*}
Therefore,
\begin{align}\label{a:Kbarlminusone}
\sum_{s=0}^{\bar l-1} (t_s K_s+1)&\leq {\frac{2\sqrt{e}}{\sqrt{2}-1}\left(\sqrt{\frac{2}{{\mu_F}}}-\sqrt{\frac{1}{\mu_0}}\right)}\ceil*{\ln \frac{32(F(x_0)-F^\star)}{\epsilon\mu_F}}+\bar l 
\end{align}
Combining~\eqref{a:tbarlbound} and~\eqref{a:Kbarlminusone} we get
\begin{align*}
\hat N&\leq 1+ \sum_{k=0}^{\bar l} (t_sK_s+1)
\\&\leq 1+\bar l+ {\frac{2\sqrt{e}}{\sqrt{2}-1}\left(\sqrt{\frac{2}{{\mu_F}}}-\sqrt{\frac{1}{\mu_0}}\right)}\ceil*{\ln \frac{32(F(x_0)-F^\star)}{\epsilon\mu_F}}\\&\qquad+\ceil*{ 2\sqrt{\frac{e}{\mu_{\bar l}}} }\ceil*{\ln \frac{2(F(x_0)-F^\star)}{\epsilon}} 
\end{align*}
 Then~\eqref{a:adapresmuogeq33} follows by noting that $\mu_{\bar l}\geq \mu_F/2$.
\end{proof}

In Table~\ref{tab:comprates} we compare the worst-case complexity bound of four algorithms which adaptively restart accelerated algorithms.  Note that the algorithms proposed by Nesterov~\cite{nesterov2013gradient} and by Lin \& Xiao~\cite{Lin2015} require strong convexity of $F$. However, the algorithm of Liu \& Yang~\cite{LiuYang17} also applies to the case when local H\"olderian error bound condition holds. The latter condition requires the existence of $\theta\geq1$ and a constant $\mu_F(L,x_0,\theta)>0$ such that
\begin{align}\label{a:holderian}
F(x)\geq F^\star +\frac{\mu_F(L,x_0,\theta)}{2}\dist_L(x, \cX_\star)^\theta, \enspace \forall x\in [F\leq F(x_0)].
\end{align}
When $\theta=2$, we recover the local quadratic growth condition. In this case, the algorithm of Liu \& Yang~\cite{LiuYang17} has a complexity bound  $\tilde O(1/\sqrt{\mu_F})$ where $\tilde O$ hides logarithm terms.

We can see  that the analysis of our algorithm leads to a worst case complexity that is $\log(1/\mu_F)$ 
times better than previous work. This will be illustrated also in the
experiment section.

\begin{table}
\centering
\begin{tabular}{|l|c|l|}
\hline Algorithm & Complexity bound & Assumption \\
\hline
Nesterov \cite{nesterov2013gradient} & $O\Big(\frac{1}{\sqrt{\mu_F}}\ln\big(\frac{1}{\mu_F}\big)\ln\big(\frac{1}{\mu_F\epsilon}\big)\Big)$ &  strong convexity\\
\hline
Lin \& Xiao \cite{Lin2015} & $O\Big(\frac{1}{\sqrt{\mu_F}}\ln\big(\frac{1}{\mu_F}\big)\ln\big(\frac{1}{\mu_F\epsilon}\big)\Big)$ &  strong convexity\\
\hline
Liu \& Yang \cite{LiuYang17} & $O\Big(\frac{1}{\sqrt{\mu_F}}\ln\big(\frac{1}{\mu_F}\big)^2\ln\big(\frac{1}{\epsilon}\big)\Big)$& H\"olderian error bound~\eqref{a:holderian} \\
\hline
This work & $O\Big(\frac{1}{\sqrt{\mu_F}}\ln\big(\frac{1}{\mu_F\epsilon}\big)\Big)$ & local quadratic error bound~\eqref{a:strconv2} \\
\hline
\end{tabular}\label{tab:comprates}
\caption{Comparison of the iteration complexity of accelerated gradient methods with adaptation to the local error bound.}
\end{table}

\subsection{A stricter test condition}
The test condition (Line~\ref{line:compute_u}) in Algorithm~\ref{alg:approxresdeter12} can be further strengthened as follows. 

For simplicity denote $\mu_F=\mu_F(L,x_0)$ and
$$
\alpha_s(\mu):=\min\left(\frac{\theta^2_{K_s-1}}{\mu}, \frac{1}{1+\frac{\mu}{2\theta^2_{K_s-1}}}\right).
$$
Let any $0\leq s'\leq s \leq \hat s$, and $1\leq t\leq t_s$. Then for the same reason as~\eqref{a:fixedine2}, we have
\begin{align*}
\|T(x_{s,t})-x_{s,t}\|^2_L\leq \frac{16}{\mu_F} \left(\frac{\theta^2_{K_s-1}}{\mu_F}\right)\left(\alpha_s(\mu_F)\right)^{t-1} \prod_{j=s-1}^{s'}\left(\alpha_j(\mu_F)\right)^{t_{j}} \|x_{s',0}-x_{s'-1,t_{s'-1}}\|_L^2.
\end{align*}
This suggests to replace Line~\ref{line:defineCs} of Algorithm~\ref{alg:approxresdeter12} by
\begin{align}\label{a:newCs}
C_s\leftarrow \frac{16}{\mu_s}\min\left\{\prod_{j=s-1}^{s'}\left(\alpha_j(\mu_s)\right)^{t_{j}} \|x_{s',0}-x_{s'-1,t_{s'-1}}\|_L^2: 0\leq s'\leq s\right\}.
\end{align}
As we  only decrease the value of $C_s$, all the theoretical analysis holds and Theorem~\ref{th:adapresdeter2} is still true with the new $C_s$ defined in~\eqref{a:newCs}. Moreover,
this change allows to identify more quickly a too large $\mu_s$ and thus can improve the  practical performance of the algorithm. 

Furthermore, if  we find that
$$\|T(x_{s,t})-x_{s,t}\|_L^2 > { C_s(\theta_{K_s-1}^2/\mu_s)^t},
$$
then before running Algorithm~\ref{FISTA} or Algorithm~\ref{APG} with $\mu_{s+1}:=\mu_s/2$, we can first do a test on $\mu_{s+1}$, i.e., check the condition
\begin{align}\notag
&\|T(x_{s,t})-x_{s,t}\|^2_L \\&\leq \min\left\{ \frac{16}{\mu_{s+1}} \left(\frac{\theta^2_{K_s-1} }{\mu_{s+1}}\right)\left(\alpha_s(\mu_{s+1})\right)^{t-1}\right. \notag\\&\qquad\qquad \left.\prod_{j=s-1}^{s'}\left(\alpha_j(\mu_{s+1})\right)^{t_{j}} \|x_{s',0}-x_{s'-1,t_{s'-1}}\|_L^2: 0\leq s'\leq s\right\}. \label{a:musplusonetest}
\end{align}
If~\eqref{a:musplusonetest} holds, then we go to Line~\ref{line:start}. Otherwise, $\mu_{s+1}$ is still too large and we decrease it further by one half.

\subsection{Looking for an $\epsilon$-solution}

Instead of an $x$ such that $\norm{T(x)-x}^2 \leq \epsilon$, 
we may be interested in an $x$ such that $F(x) - F^\star \leq \epsilon'$, that is an $\epsilon'$-solution.

In view of \eqref{a:FTLy1}, if $\norm{x - T(x)}_L^2 \leq \frac{\epsilon' \mu_F(L, x)}{8}$, then $F(T(x)) - F^\star \leq \epsilon'$.
As a result, except from the fact that we cannot terminate the algorithm in Line~\ref{line:stop_test},
Algorithm~\ref{alg:approxresdeter12} is applicable with $\epsilon = \frac{\epsilon' \mu_F(L, x)}{8}$.

We then will obtain an $\epsilon'$-solution after a number of iterations at most equal to 
\begin{align*}
& 1+ \ceil*{\frac{2\sqrt{e}}{\sqrt{2}-1}\left(\sqrt{\frac{2}{{\mu_F(L,x_0)}}}-\sqrt{\frac{1}{\mu_0}}\right)}\ceil*{\ln \frac{2^8(F(x_0)-F^\star)}{\epsilon'\mu_F(L,x_0)^2}}\notag \\&\qquad +\ceil*{ {\frac{16\sqrt{2e}}{\sqrt{\mu_F(L,x_0)}}} }\ceil*{\ln \frac{2(F(x_0)-F^\star)}{\epsilon' \mu_F(l, x_0)}} 
\end{align*}
Note that compared to the result of Theorem~\ref{th:adapresdeter2}, we only add a constant factor.

\section{Numerical experiments}
\label{sec:expe}

In this section we present some numerical results to demonstrate the effectiveness of the proposed algorithms.
We apply Algorithm~\ref{alg:approxresdeter12} to solve regression and classification problems which typically take the form of
\begin{align}\label{a:gA}
\min_{x\in \R^n} F(x):= g(Ax)+\psi(x)
\end{align}
where $A\in \R^{m\times n}$, $g:\R^m\rightarrow \R$ has Lipschitz continuous gradient and $\psi: \R^n\rightarrow \R\cup\{ +\infty\}$ is simple. The model includes in particular the 
$L^1$-regularized least squares problem (Lasso) 
and the $L^1$-$L^2$-regularized logistic regression problem. Note that the following problem is dual  to~\eqref{a:gA},
$$
\max_{y\in \R^m} G(y):=-\psi^*(A^\top y)-g^*(-y),
$$ 
where $g^*$ (resp. $\psi^*$) denotes the convex conjugate function of $g$ (resp. $\psi$).
We define the primal dual gap associated to a point $x\in \dom (\psi)$ as:
\begin{align}\label{a:primaldualgap}
F(x)-G(-\alpha(x) A^\top\nabla g(Ax))
\end{align}
where $\alpha(x) \in [0,1]$ is chosen as the largest $\alpha \in [0,1]$ such that 
$G(-\alpha A^\top\nabla g(Ax))<+\infty$.
Note that $x\in \dom(\psi)$ is an optimal solution of~\eqref{a:gA} if and only if the associated primal dual gap~\eqref{a:primaldualgap} equals 0.

We compare five methods: Gradient Descent (GD),  FISTA~\cite{beck2009fista}, AdapAPG~\cite{Lin2015}, AdaAGC~\cite{LiuYang17}, and AdaRES (Algorithm~\ref{alg:approxresdeter12} using FISTA in Line~\ref{line:fistaorapg}). We plot the primal dual gap~\eqref{a:primaldualgap} versus running time. Note that GD and FISTA do not depend on the initial guess of the value $\mu_F(L,x_0)$.

\subsection{Lasso problem}

We  present in Figure~\ref{fig:cputime} the  experimental results for  solving the $L^1$-regularised least squares problem (Lasso):
\begin{align}\label{a:lasso}
\min_{x \in \mathbb{R}^n} \frac{1}{2} \norm{A x - b}^2_2 + \frac{\|A^\top b\|_{\infty}}{\lambda_1} \norm{x}_1,
\end{align}
on the dataset cpusmall$\_$scale with $n=12$ and $m=8192$. The value $L$ is set to be $\tr(A^\top A)$.
We test with $\lambda_1=10^4$, $\lambda_1=10^5$ and $\lambda_1=10^6$. For each value of $\lambda_1$, we vary the initial guess $\mu_0$ from 0.1 to $10^{-5}$.  Compared with AdaAPG and AdaAGC,   Algorithm~\ref{alg:approxresdeter12} seems to be more efficient and less sensitive to the guess of $\mu_F$.

\subsection{Logistic regression}

We present in Figure~\ref{fig:dorothea} the  experimental results for solving the $L^1$-$L^2$ regularized logistic regression problem:
\begin{align}\label{alr}
\min_{x \in \mathbb{R}^n} \frac{\lambda_1}{2 \|A^\top b\|_{\infty}} \sum_{j=1}^m \log(1 + \exp(b_j a_j^\top x)) + \norm{x}_1+\frac{\lambda_2}{2}\|x\|^2,\end{align}
on the dataset dorothea with $n=100,000$ and $m=800$.
In our experiments, we set
$$
\lambda_2=\frac{L}{ 10n}
$$
where $$
L:=\frac{\lambda_1}{ 8\|A^\top b\|_{\infty}}\sum_{j=1}^n\sum_{i=1}^m (b_jA_{ij})^2
$$
is an upper bound of the Lipschitz constant of the function 
$$
f(x)\equiv \frac{\lambda_1}{ 2\|A^\top b\|_{\infty}} \sum_{j=1}^m \log(1 + \exp(b_j a_j^\top x)).
$$
Thus $\mu_F\geq 1/(10n)=10^{-6}$. We test with $\lambda_1=10$, $\lambda_1=100$ and $\lambda_1=1000$. For each value of $\lambda_1$, we vary the initial guess $\mu_0$ from 0.01 to $10^{-6}$. On this problem Algorithm~\ref{alg:approxresdeter12} also outperforms AdaAPG and AdaAGC in all the cases.

\newcommand{\sizefigure}{0.29\linewidth}
\afterpage{
\clearpage
\thispagestyle{empty}

\begin{figure}[htbp]
\centering
\includegraphics[width=\sizefigure]{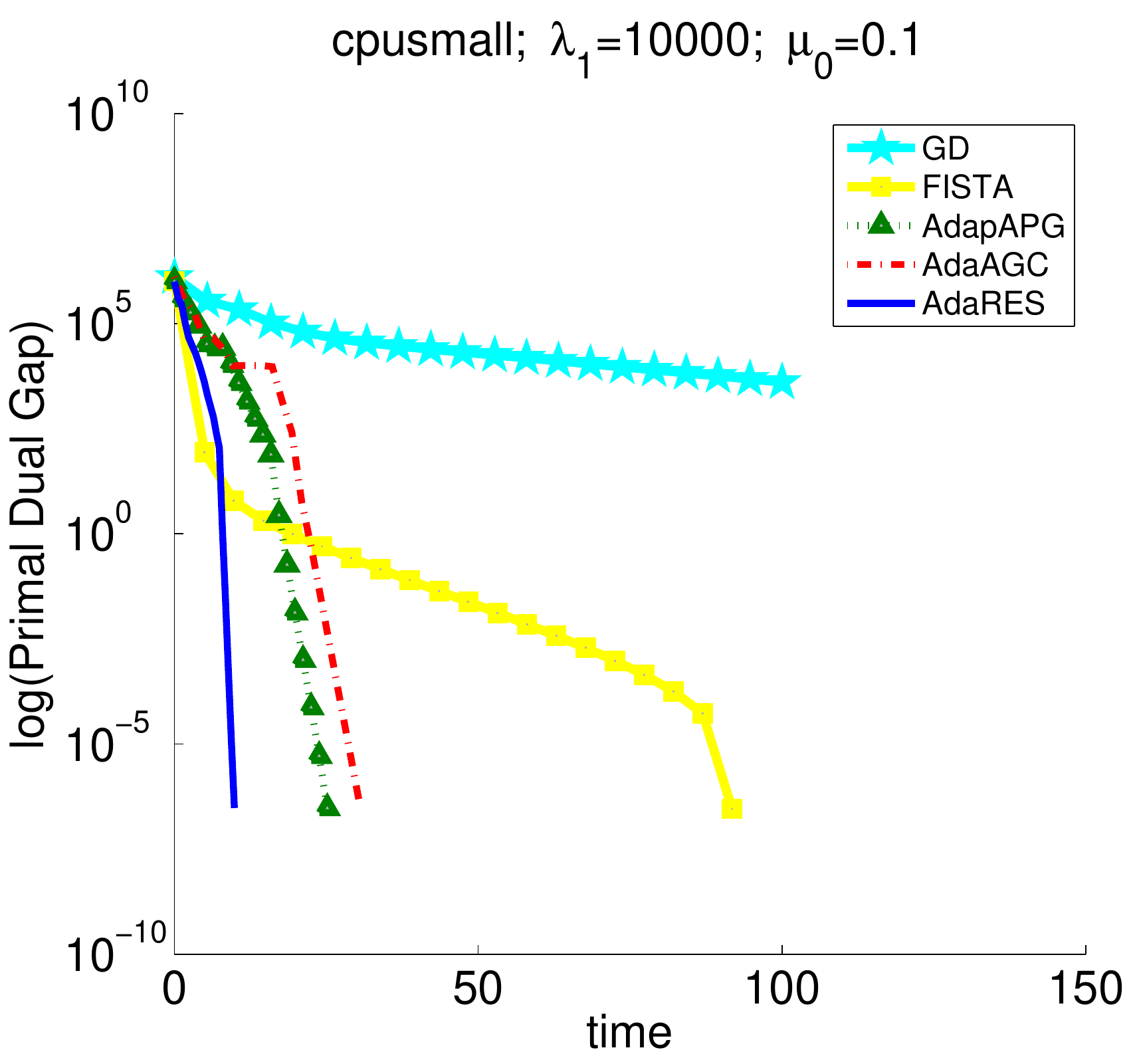}
\includegraphics[width=\sizefigure]{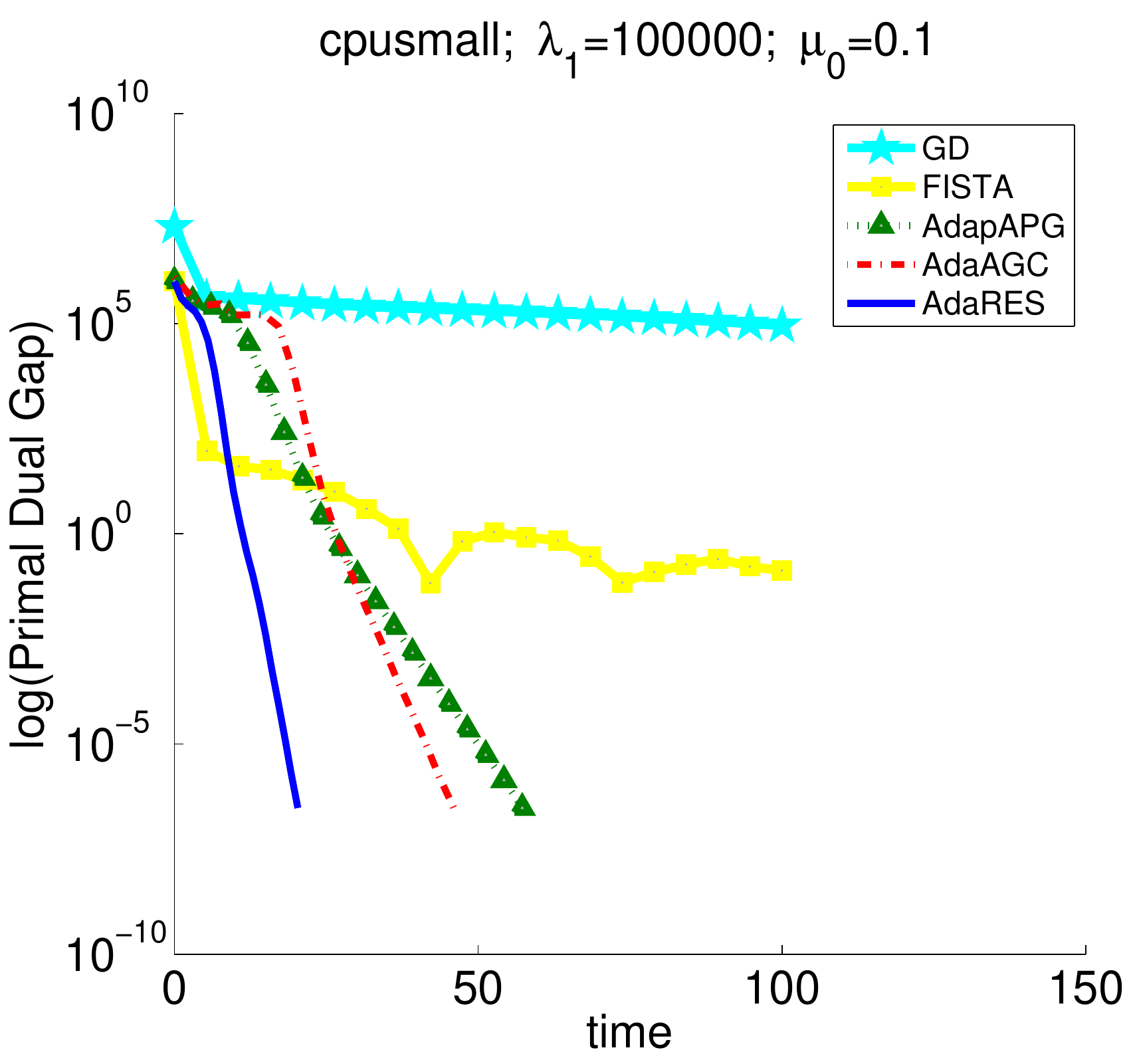}
\includegraphics[width=\sizefigure]{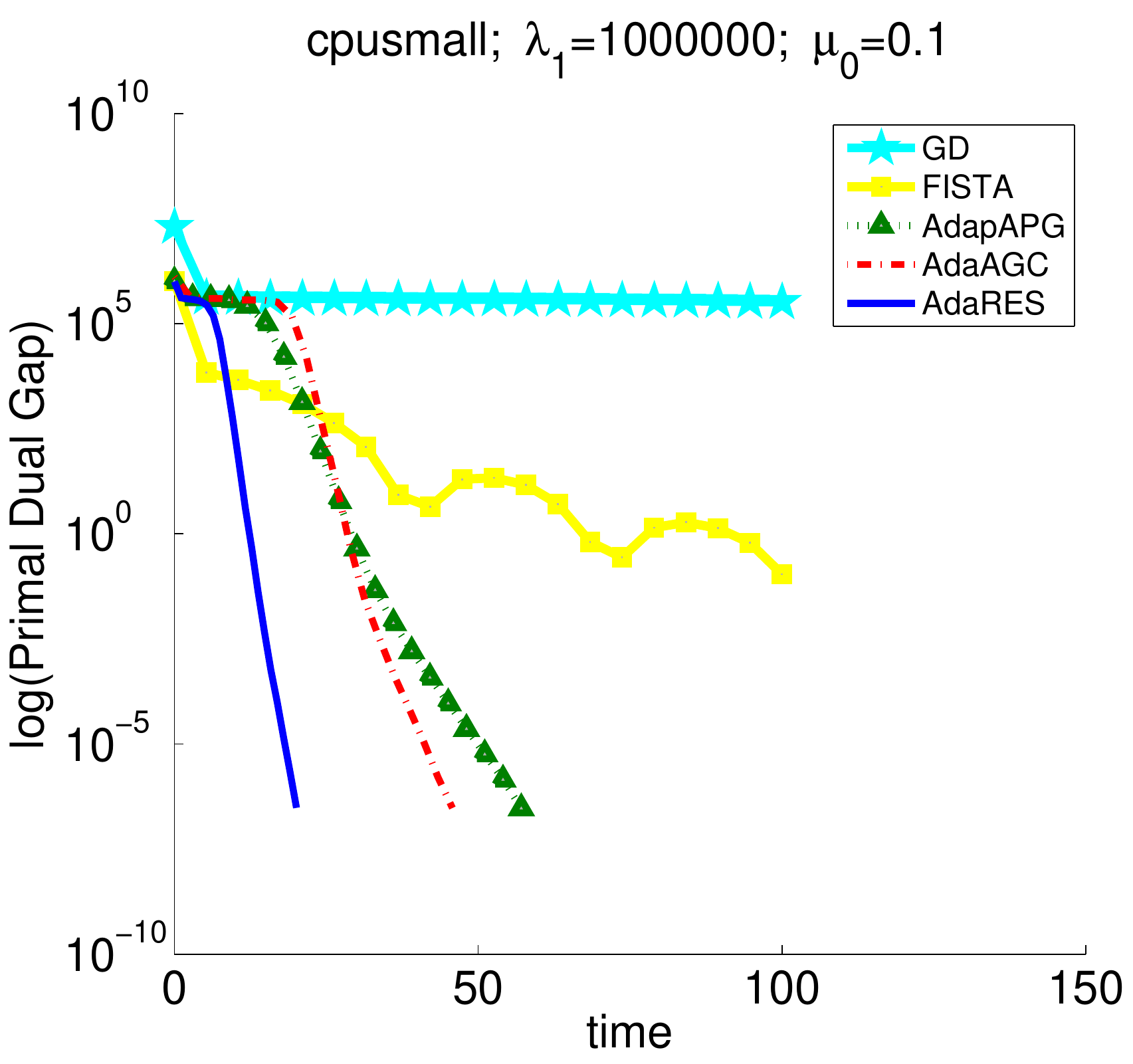}
\includegraphics[width=\sizefigure]{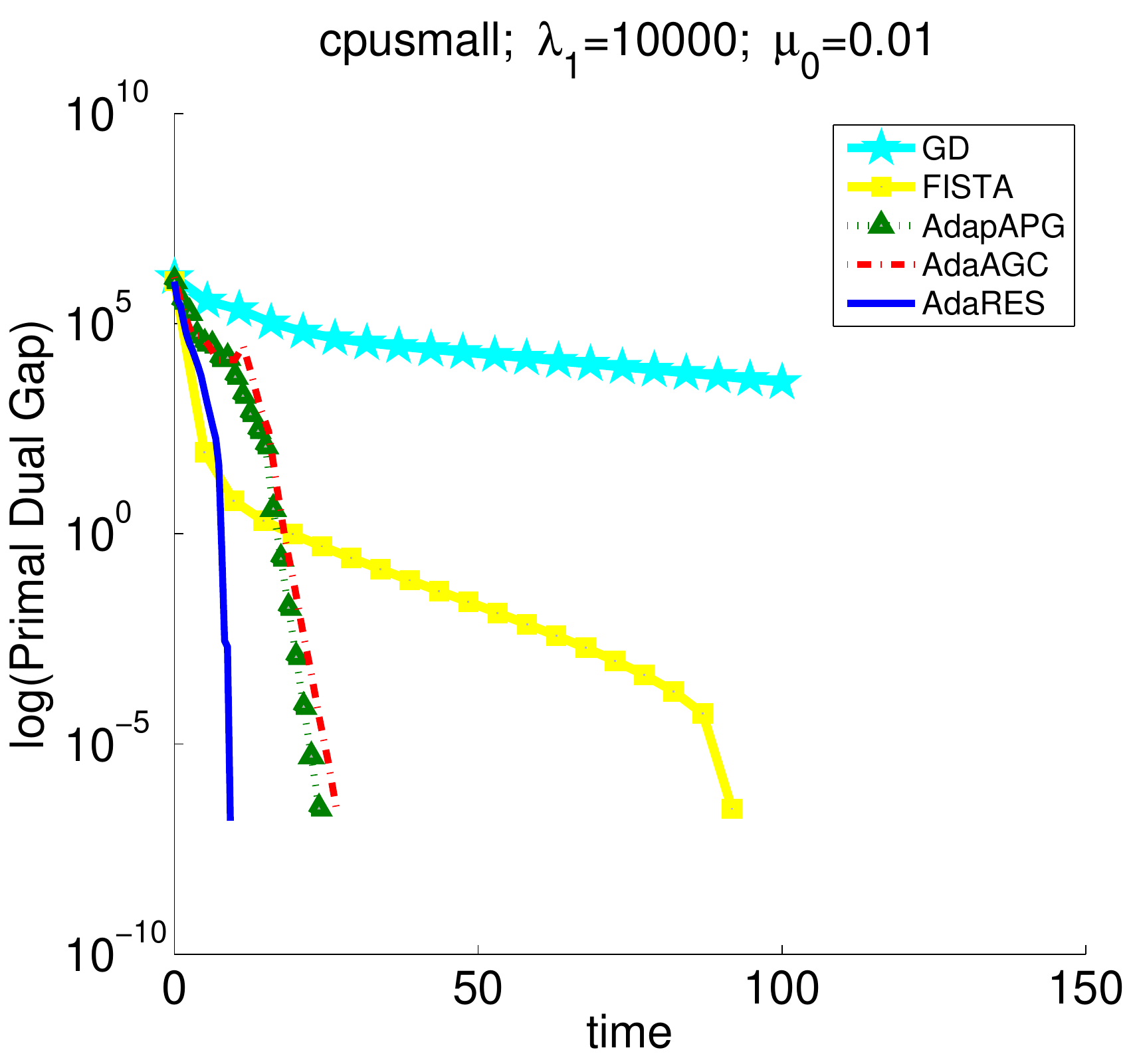}
\includegraphics[width=\sizefigure]{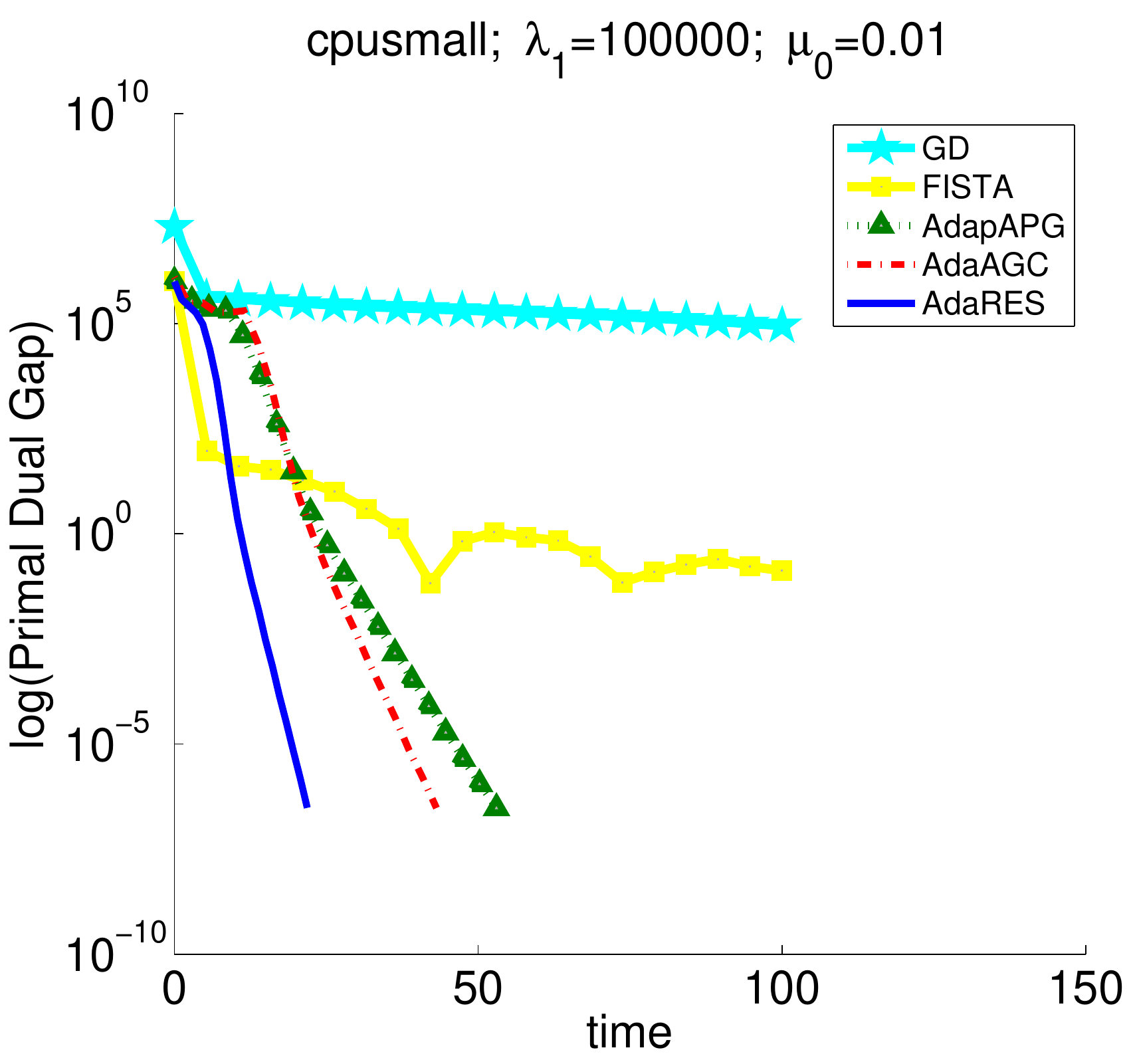}
\includegraphics[width=\sizefigure]{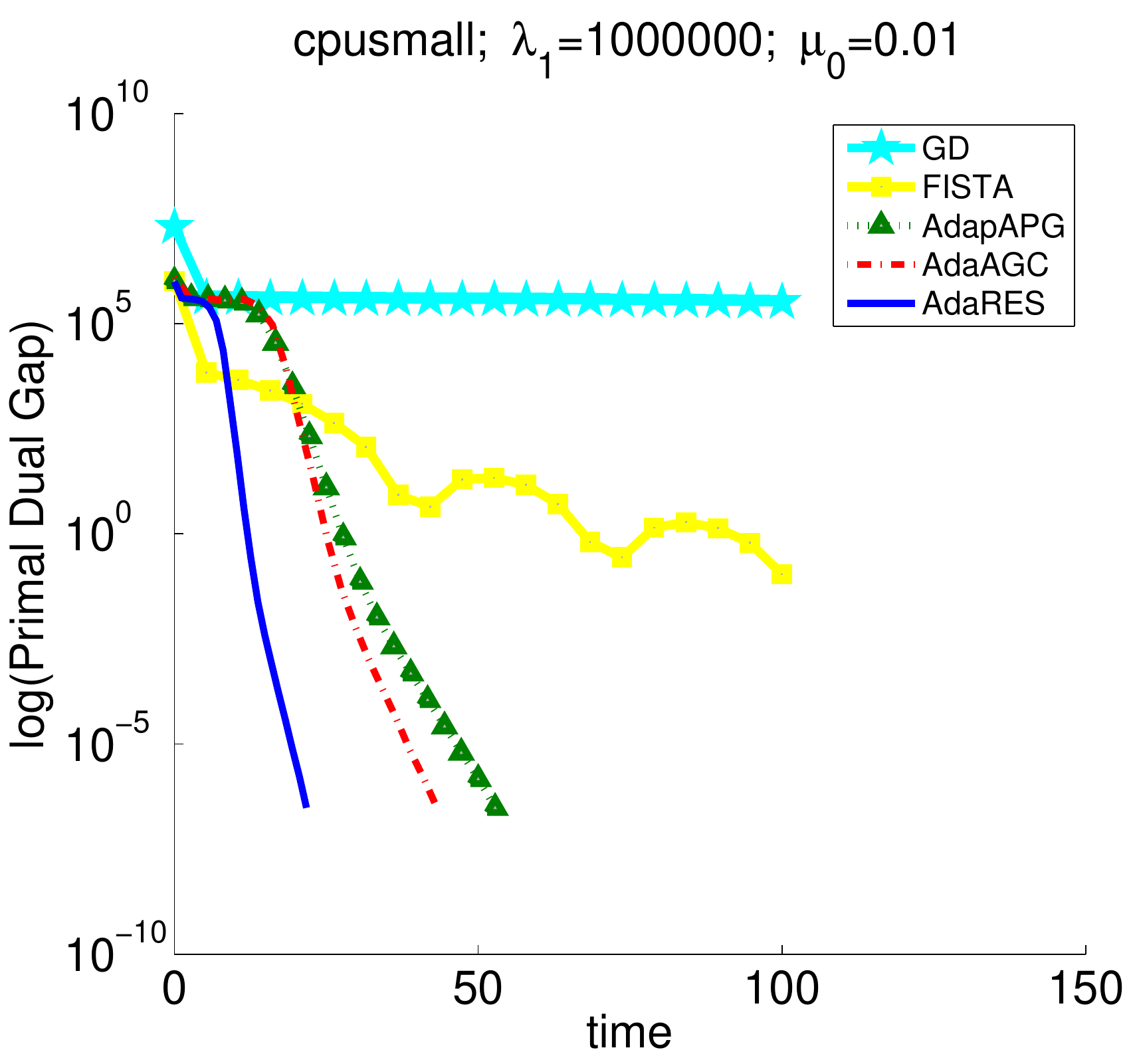}
\includegraphics[width=\sizefigure]{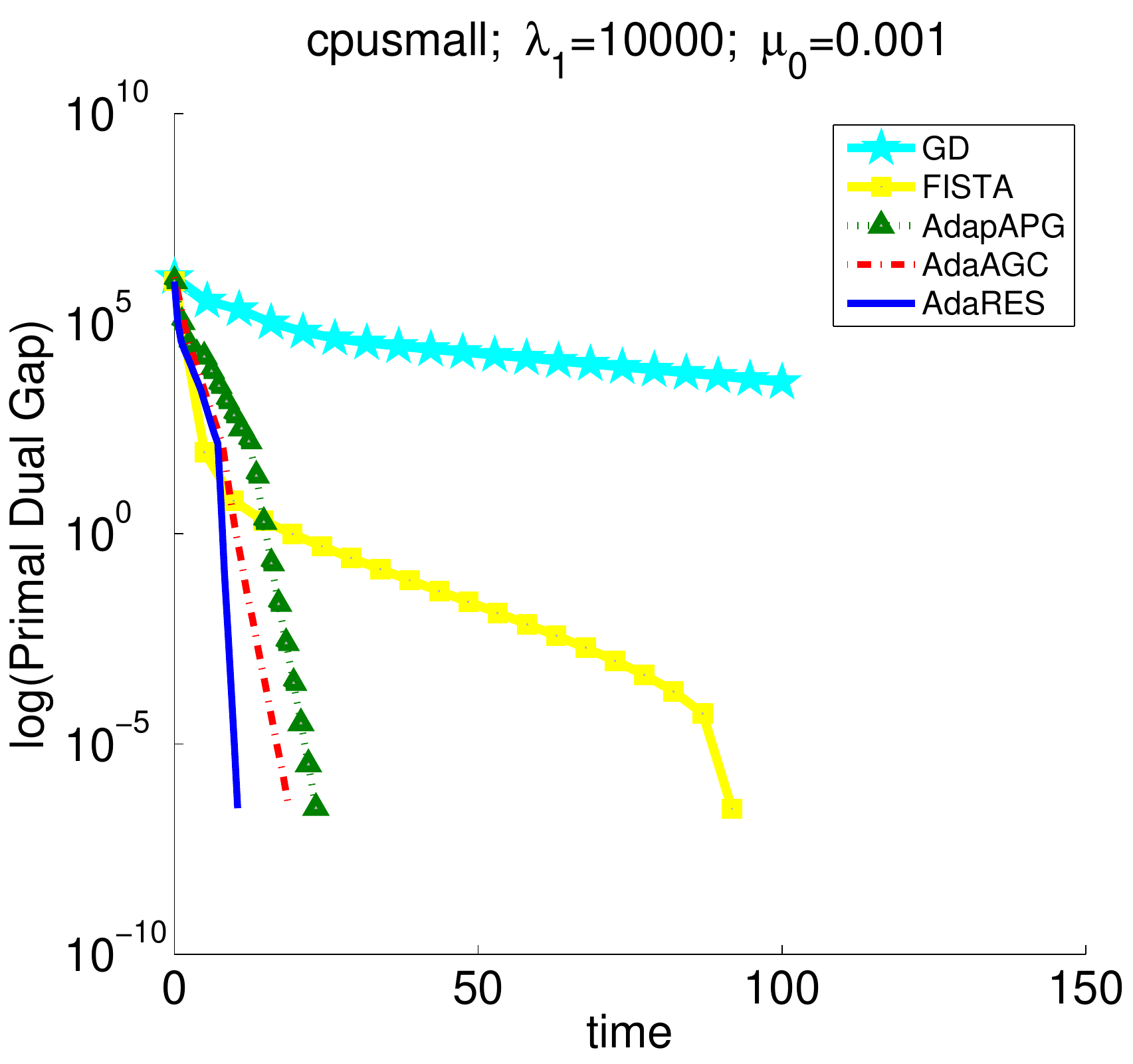}
\includegraphics[width=\sizefigure]{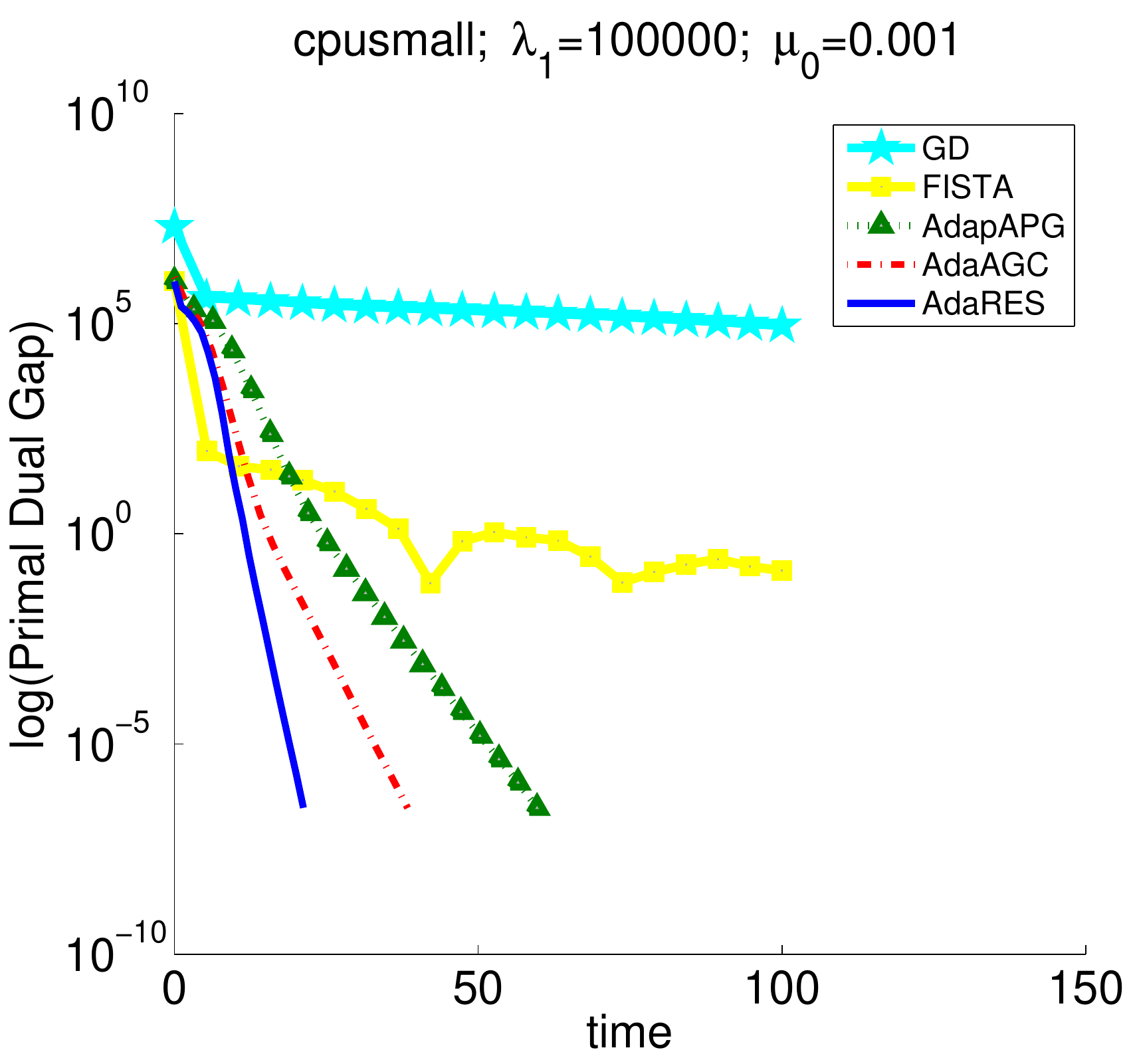}
\includegraphics[width=\sizefigure]{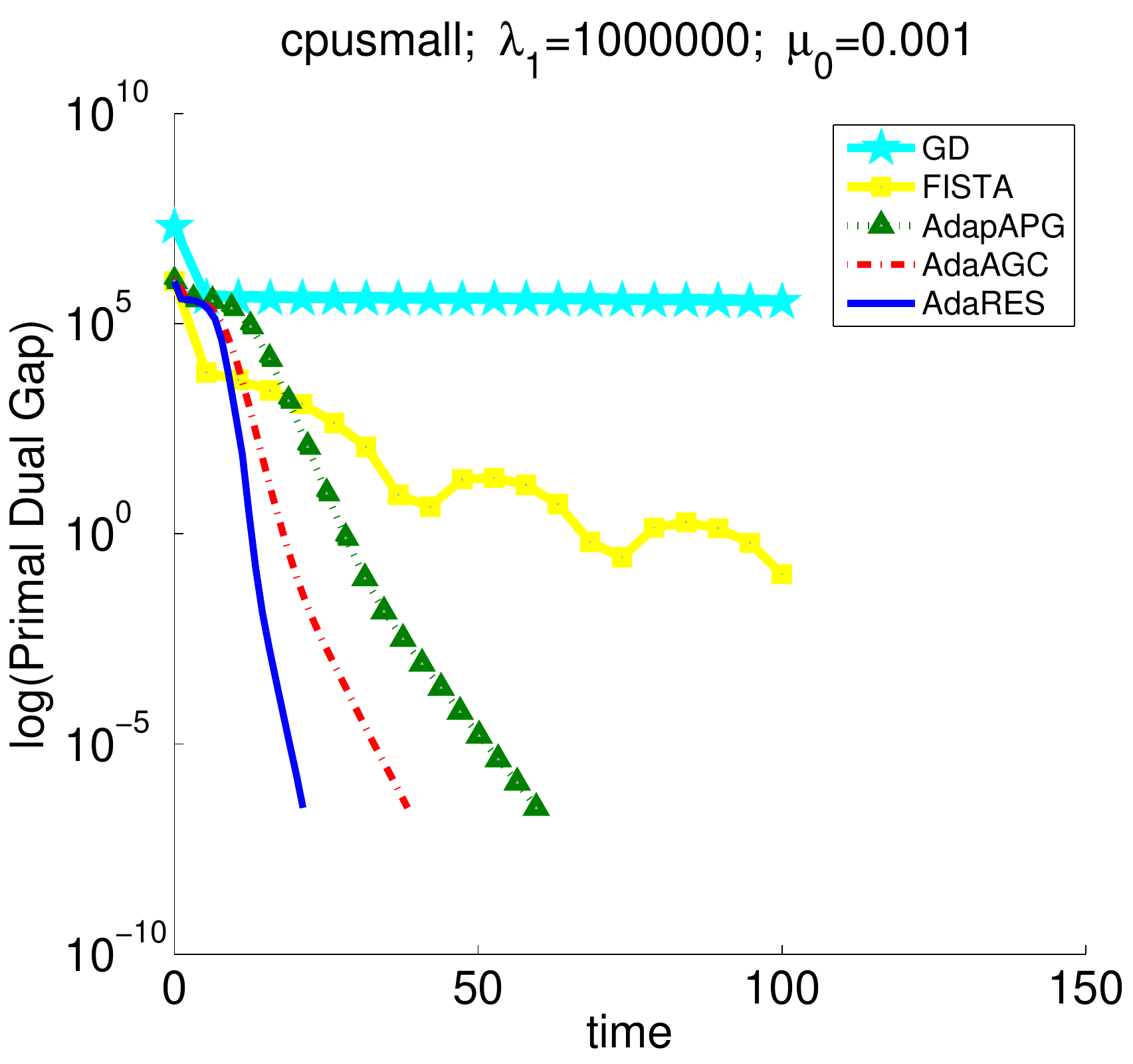}
\includegraphics[width=\sizefigure]{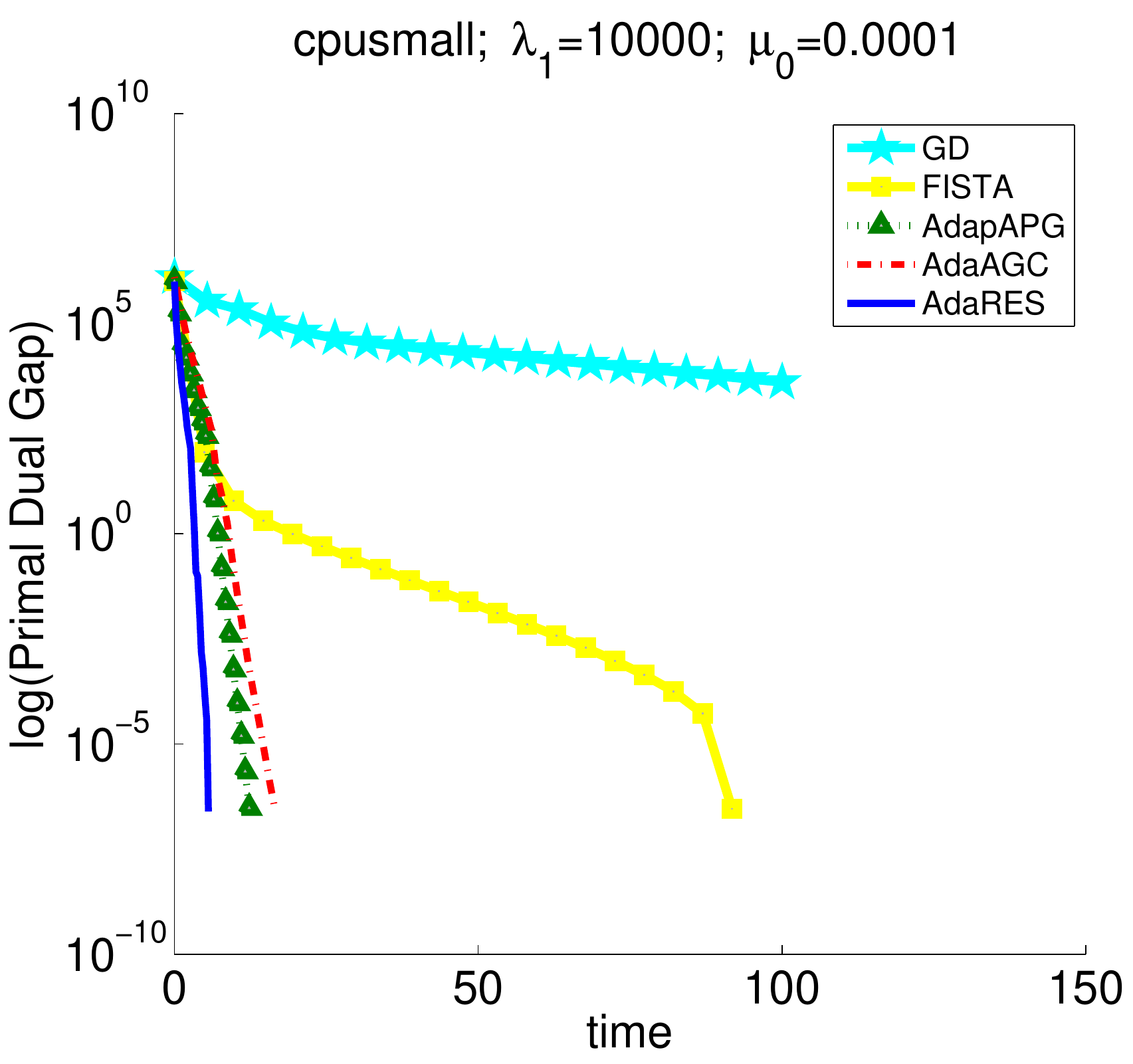}
\includegraphics[width=\sizefigure]{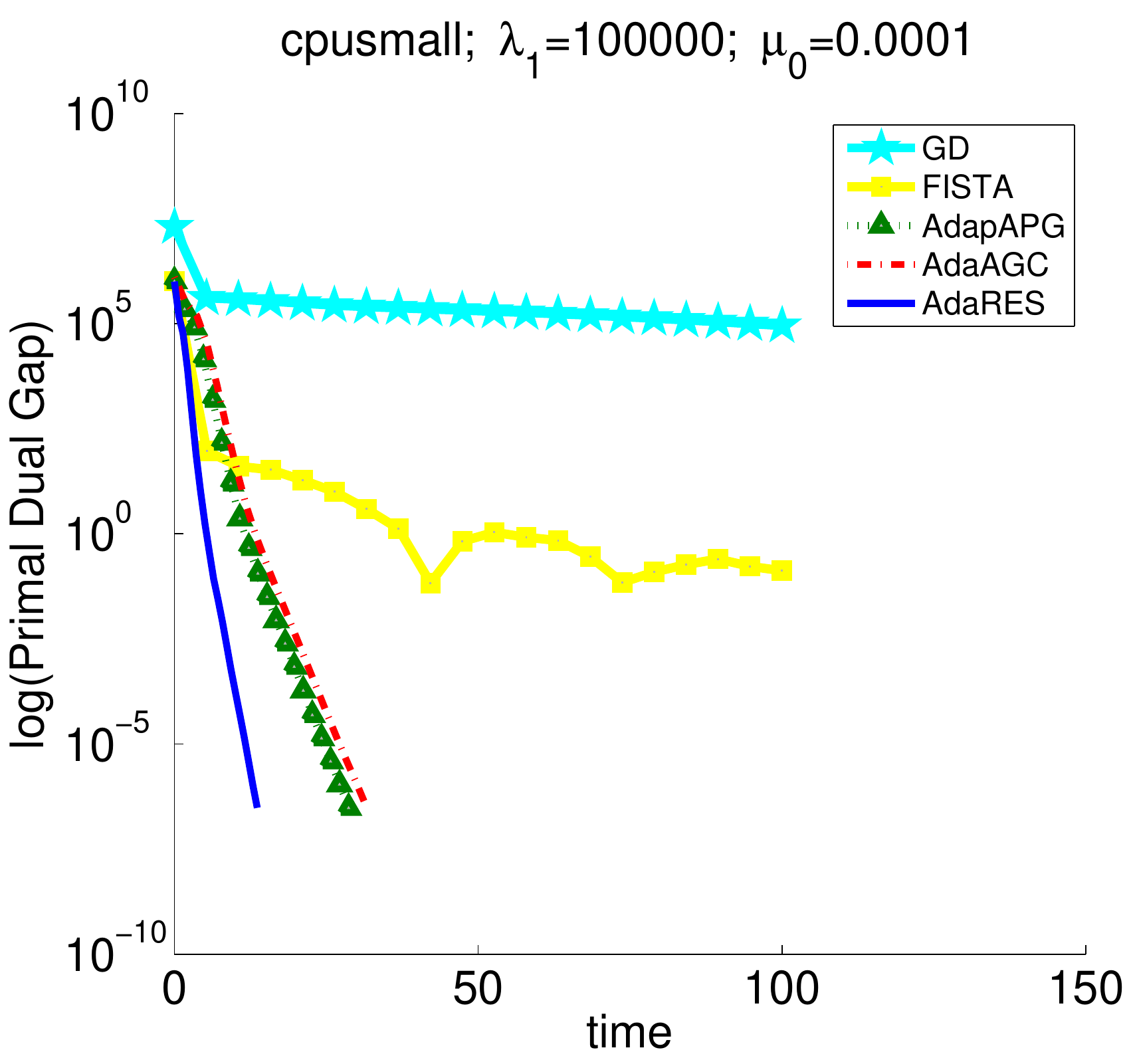}
\includegraphics[width=\sizefigure]{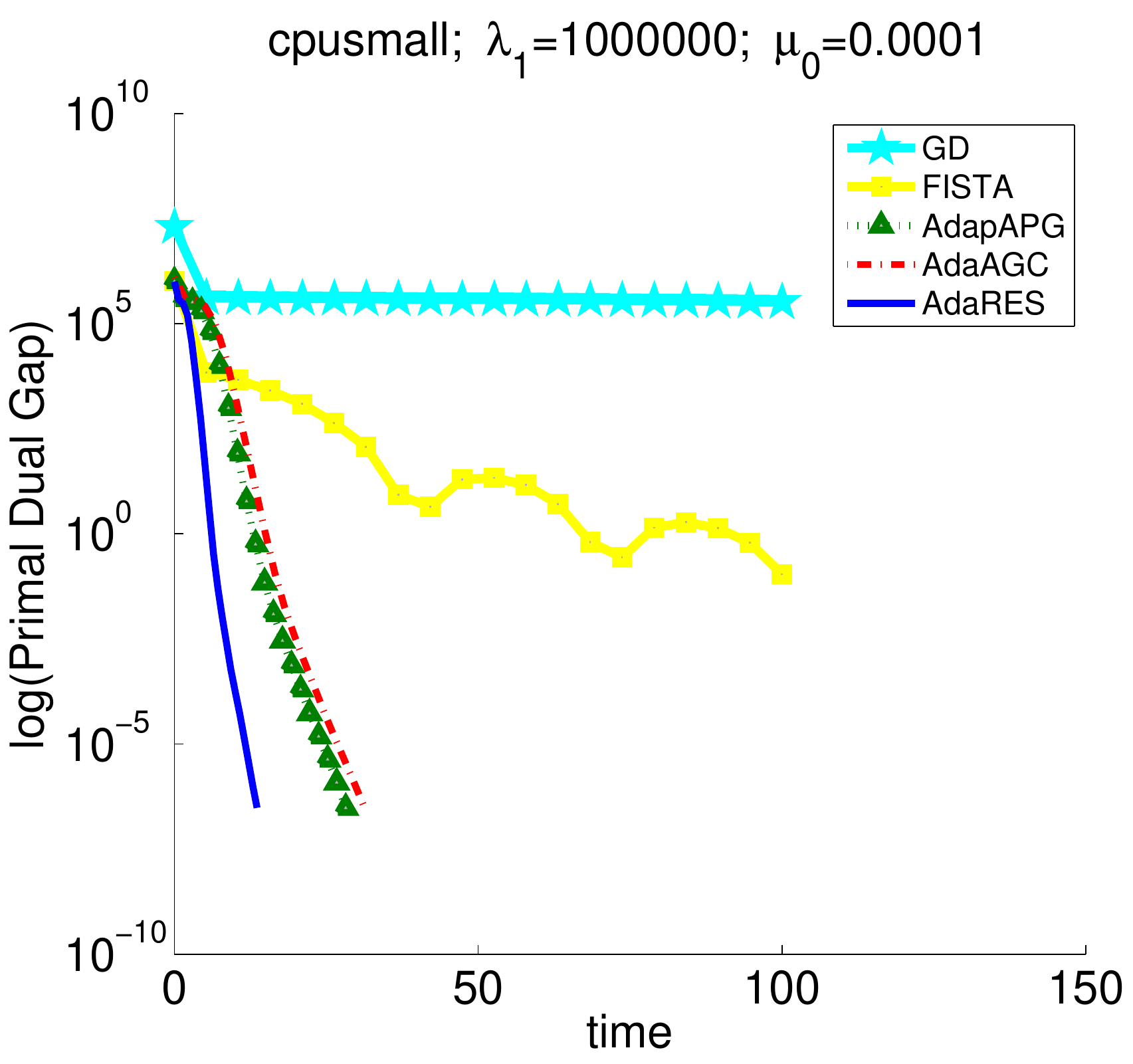}
\includegraphics[width=\sizefigure]{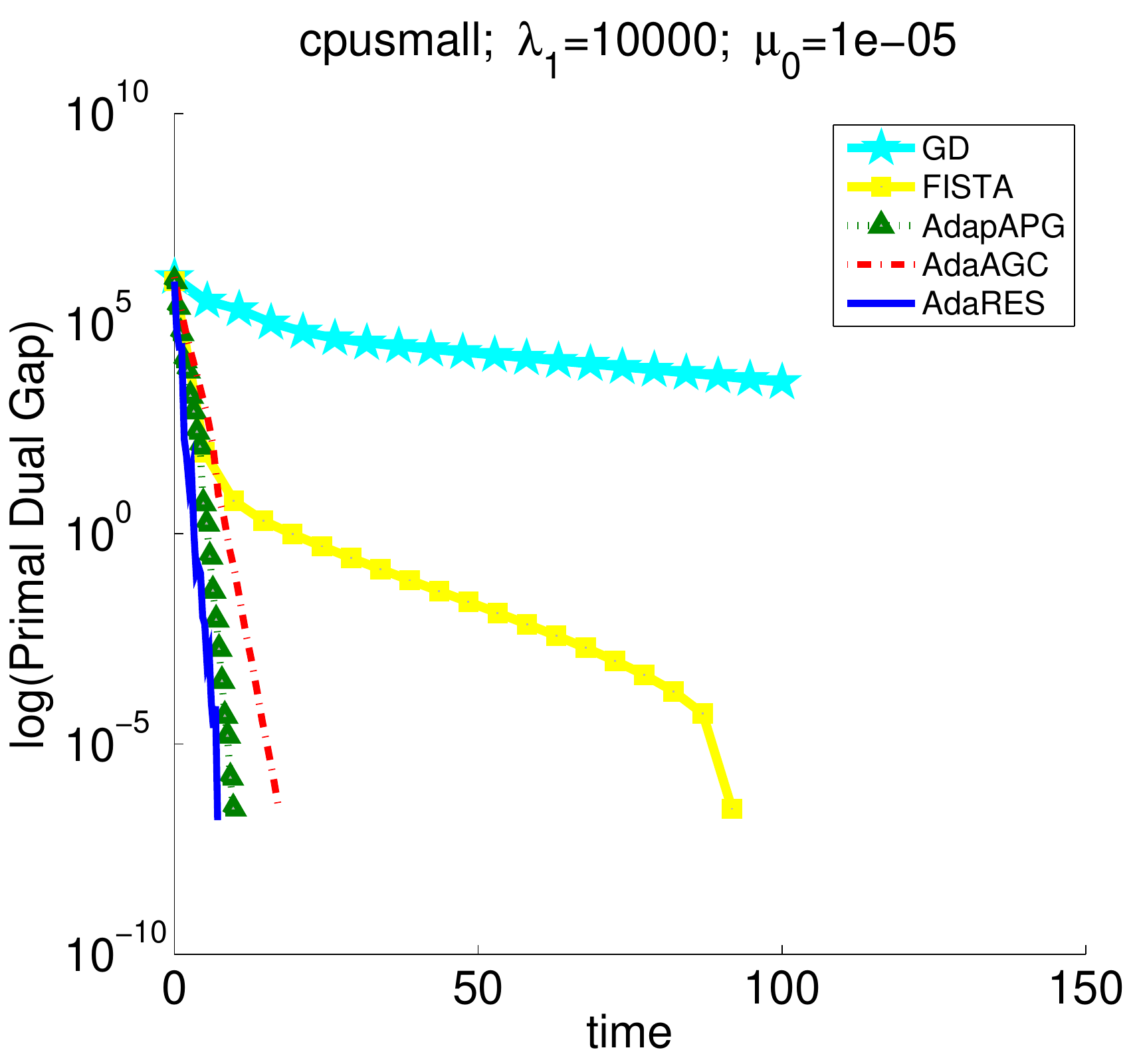}
\includegraphics[width=\sizefigure]{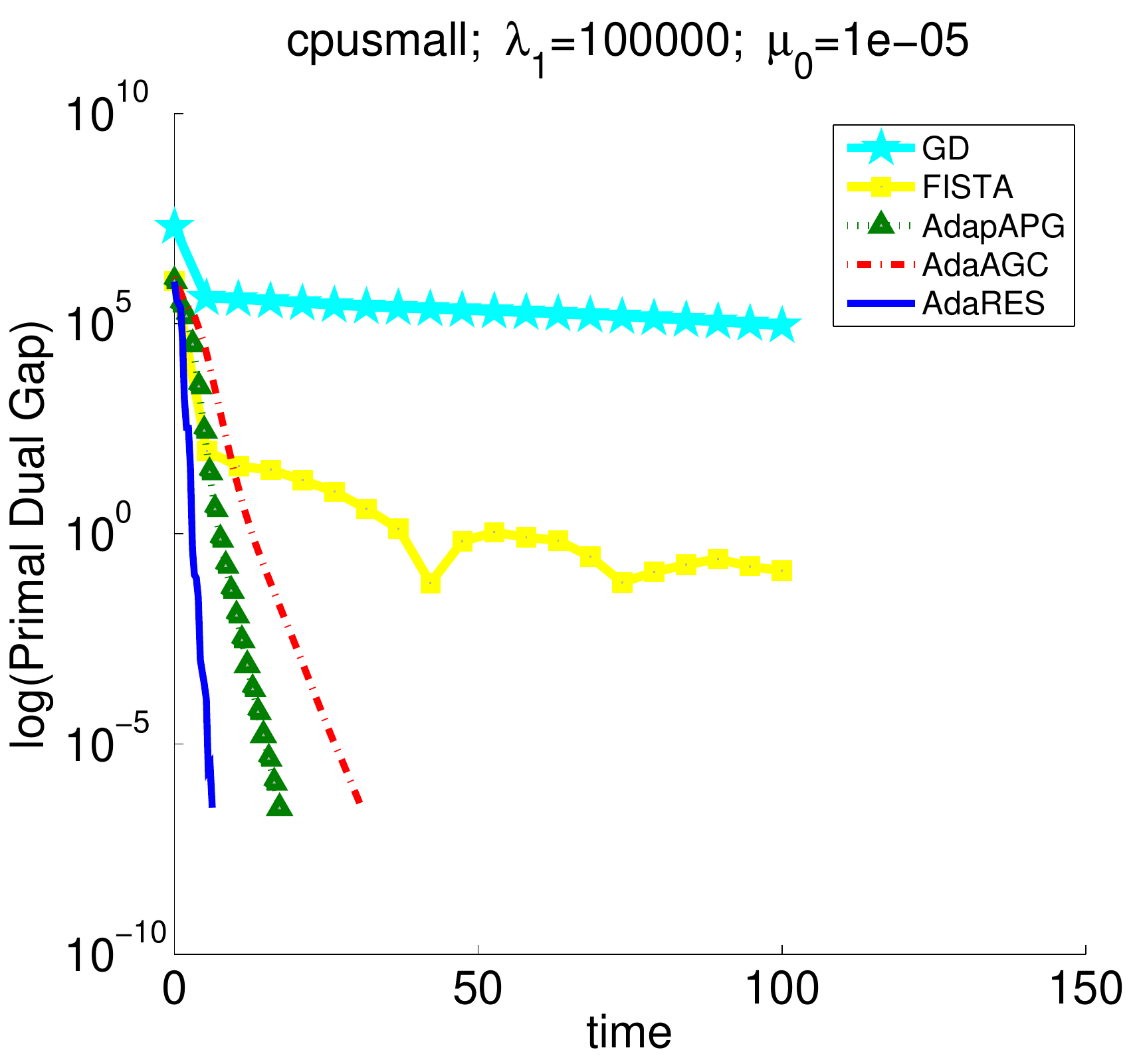}
\includegraphics[width=\sizefigure]{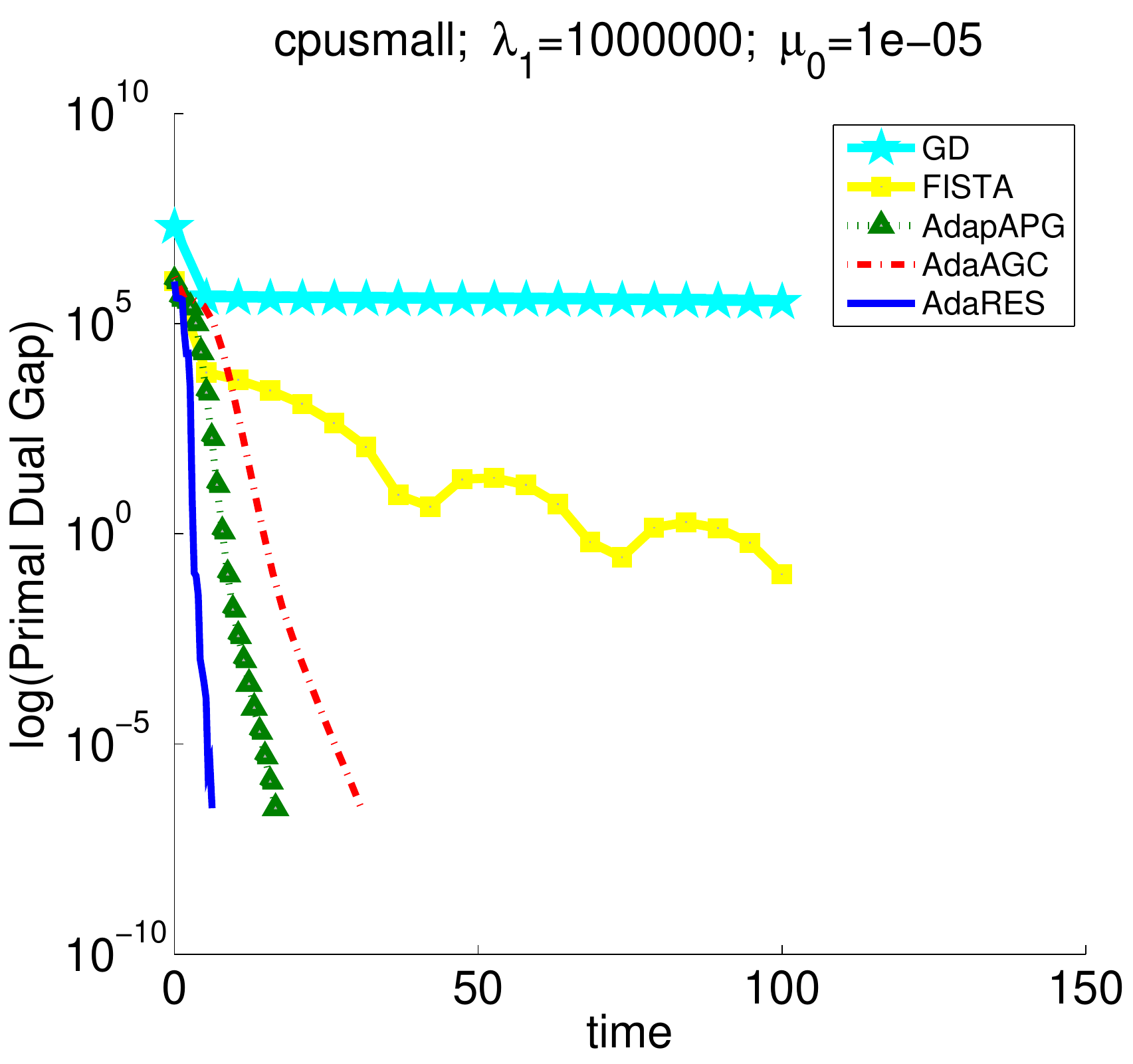}
\caption{Experimental results on the Lasso problem~\eqref{a:lasso} and the dataset cpusmall. Column-wise: we solve the same problem with a different 
a priori on the quadratic error bound. Row-wise: we use the same a priori
on the quadratic error bound but the weight of the 1-norm is varying.
}
\label{fig:cputime}
\end{figure}

\clearpage
}

\afterpage{
\clearpage
\thispagestyle{empty}
\begin{figure}[htbp]
\centering
\includegraphics[width=\sizefigure]{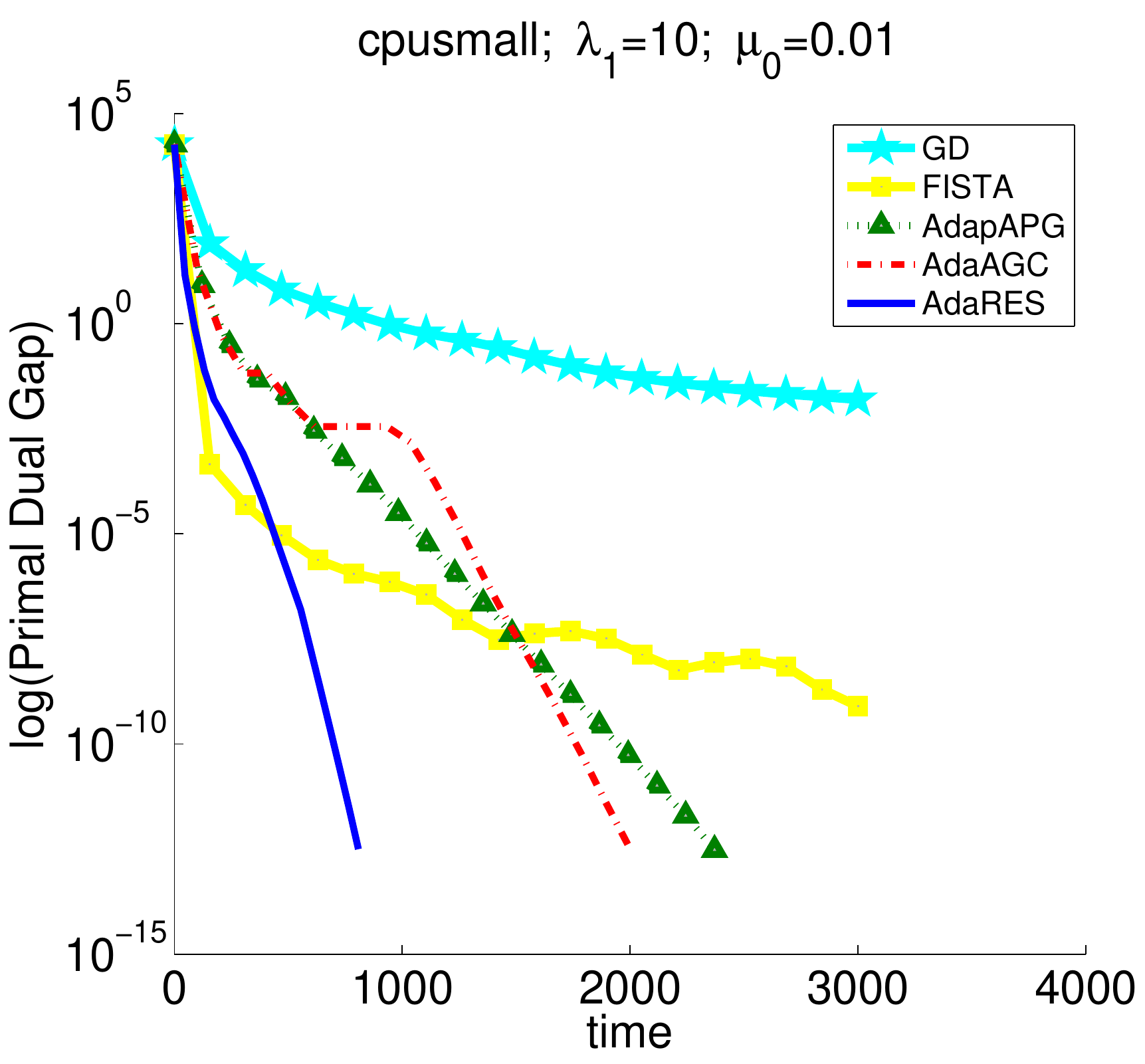}
\includegraphics[width=\sizefigure]{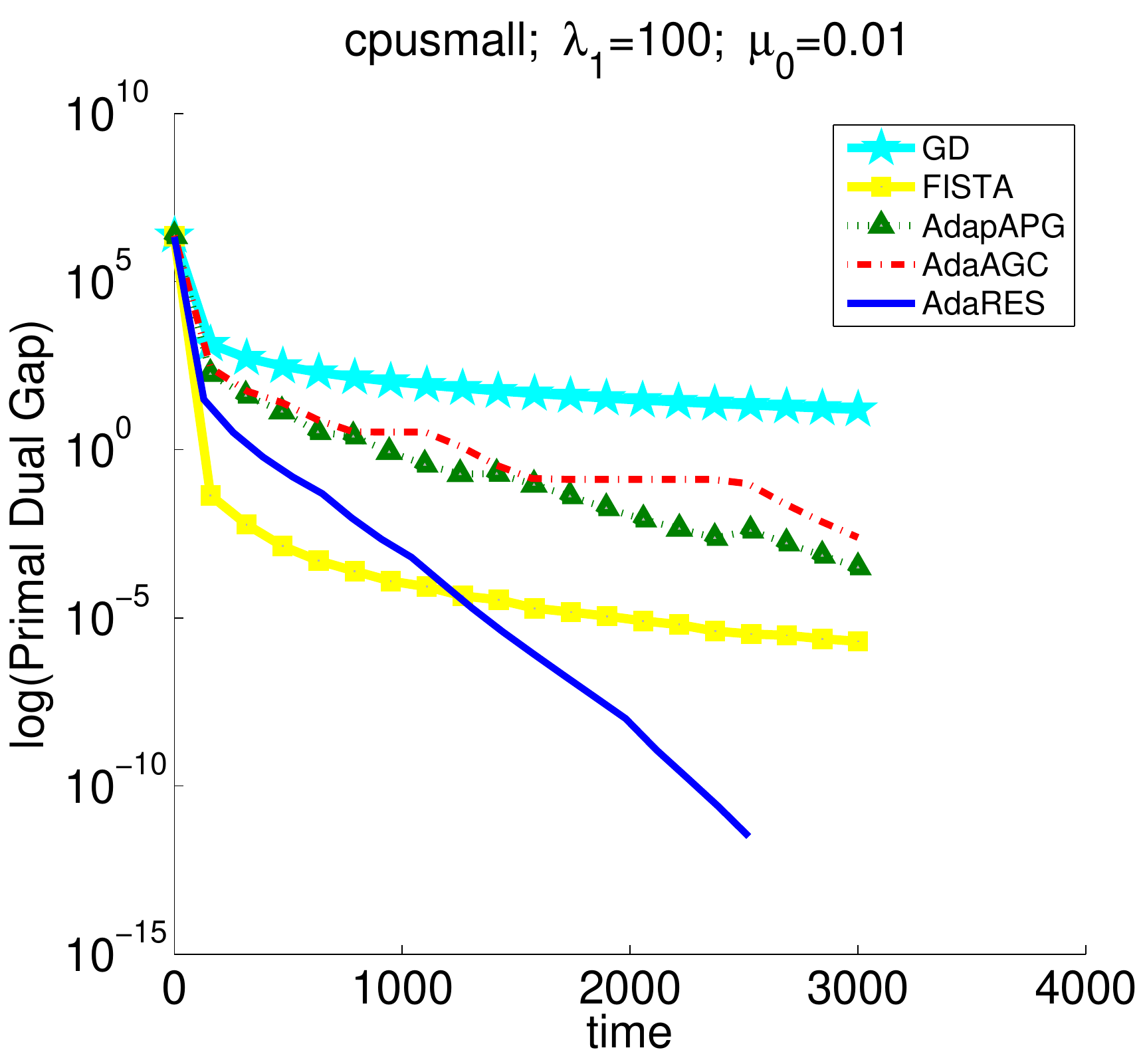}
\includegraphics[width=\sizefigure]{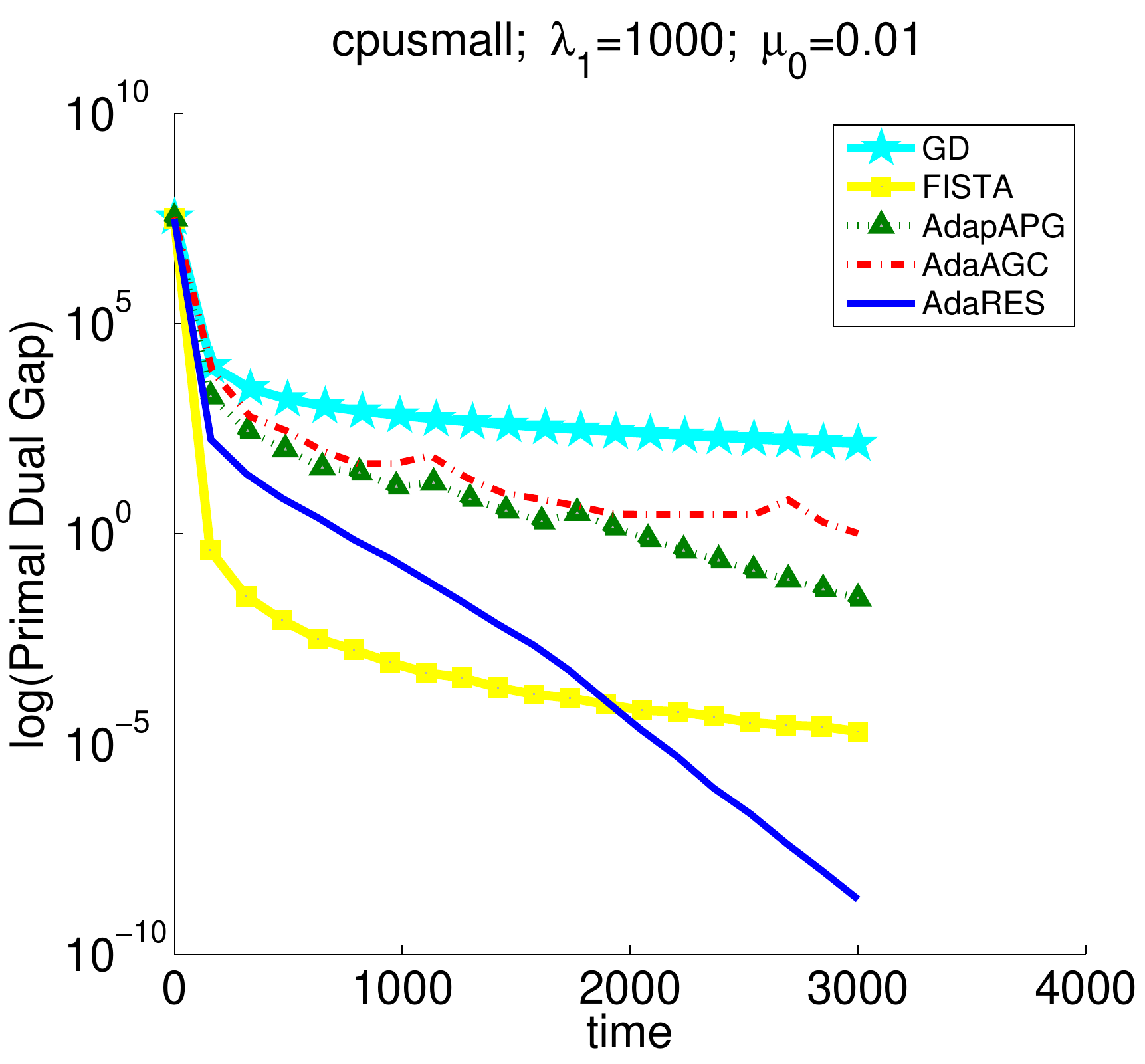}
\includegraphics[width=\sizefigure]{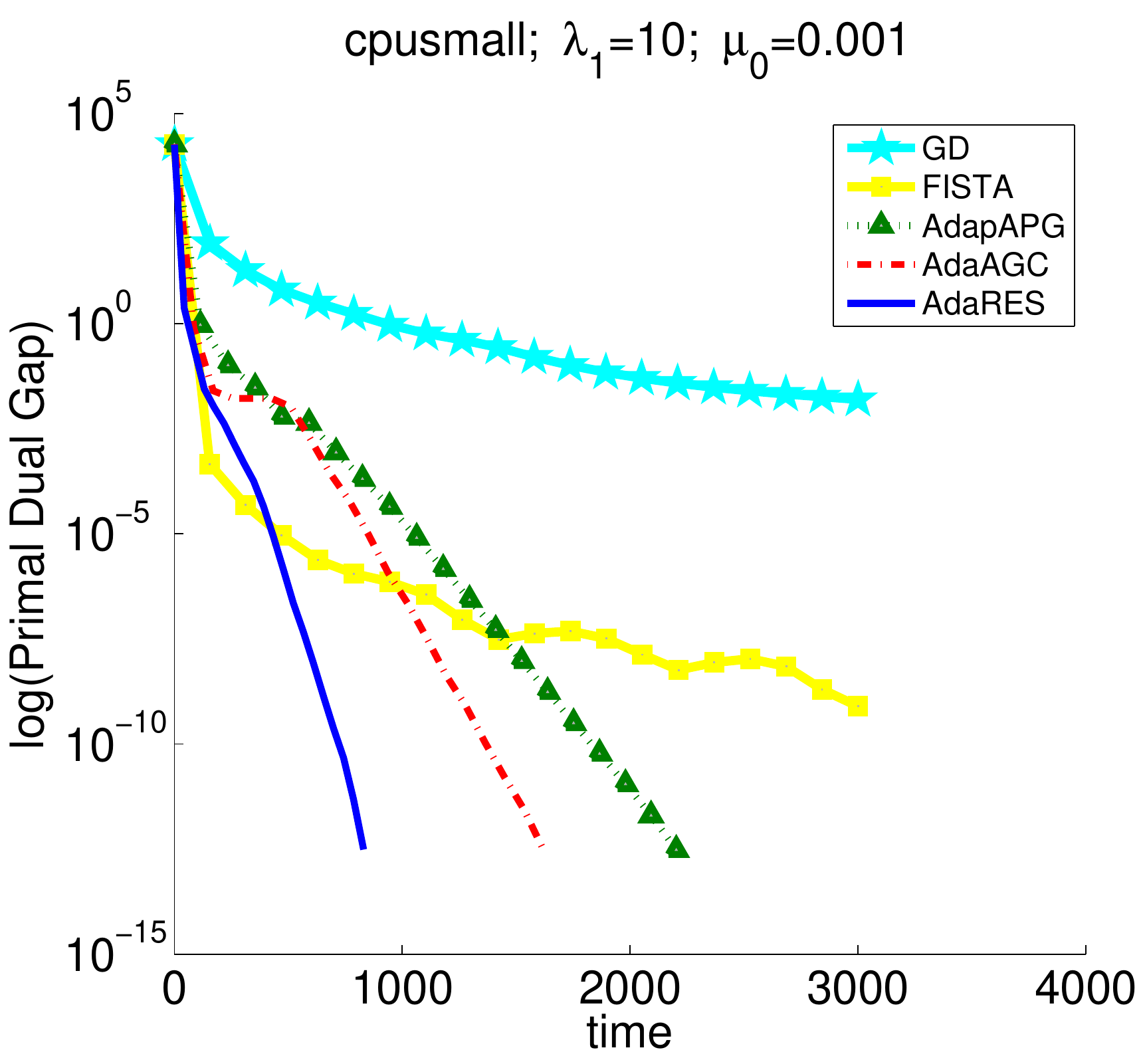}
\includegraphics[width=\sizefigure]{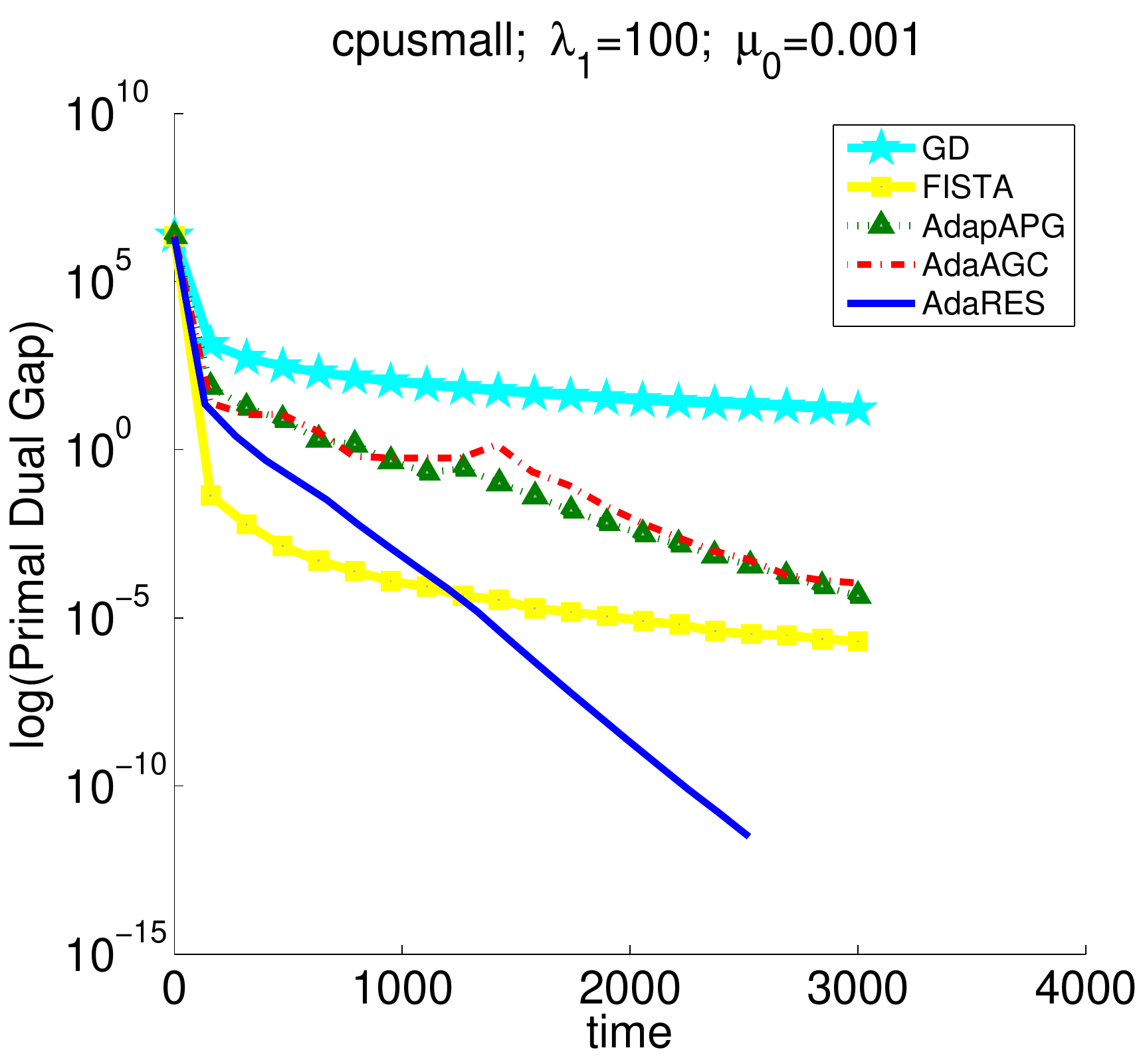}
\includegraphics[width=\sizefigure]{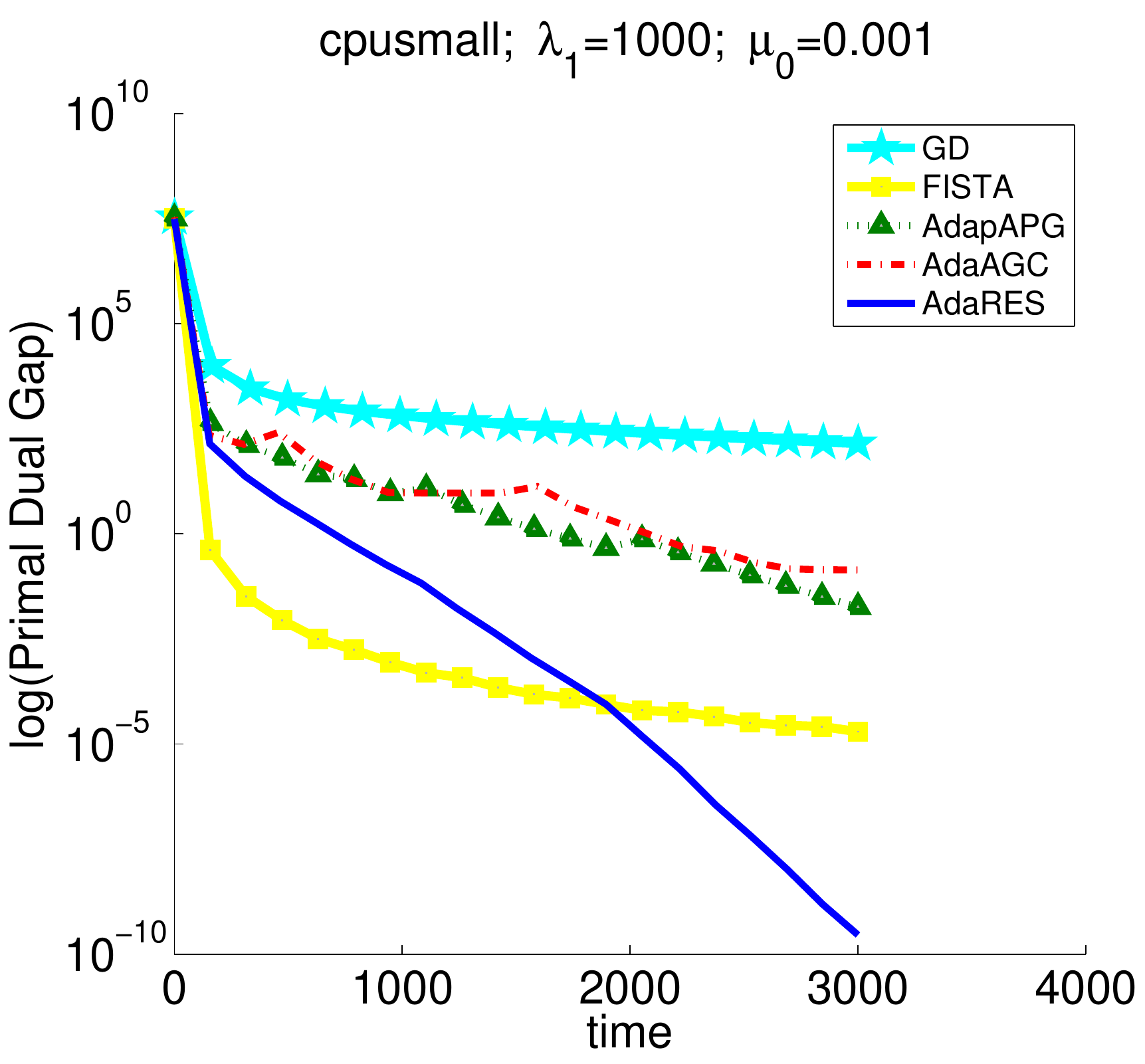}
\includegraphics[width=\sizefigure]{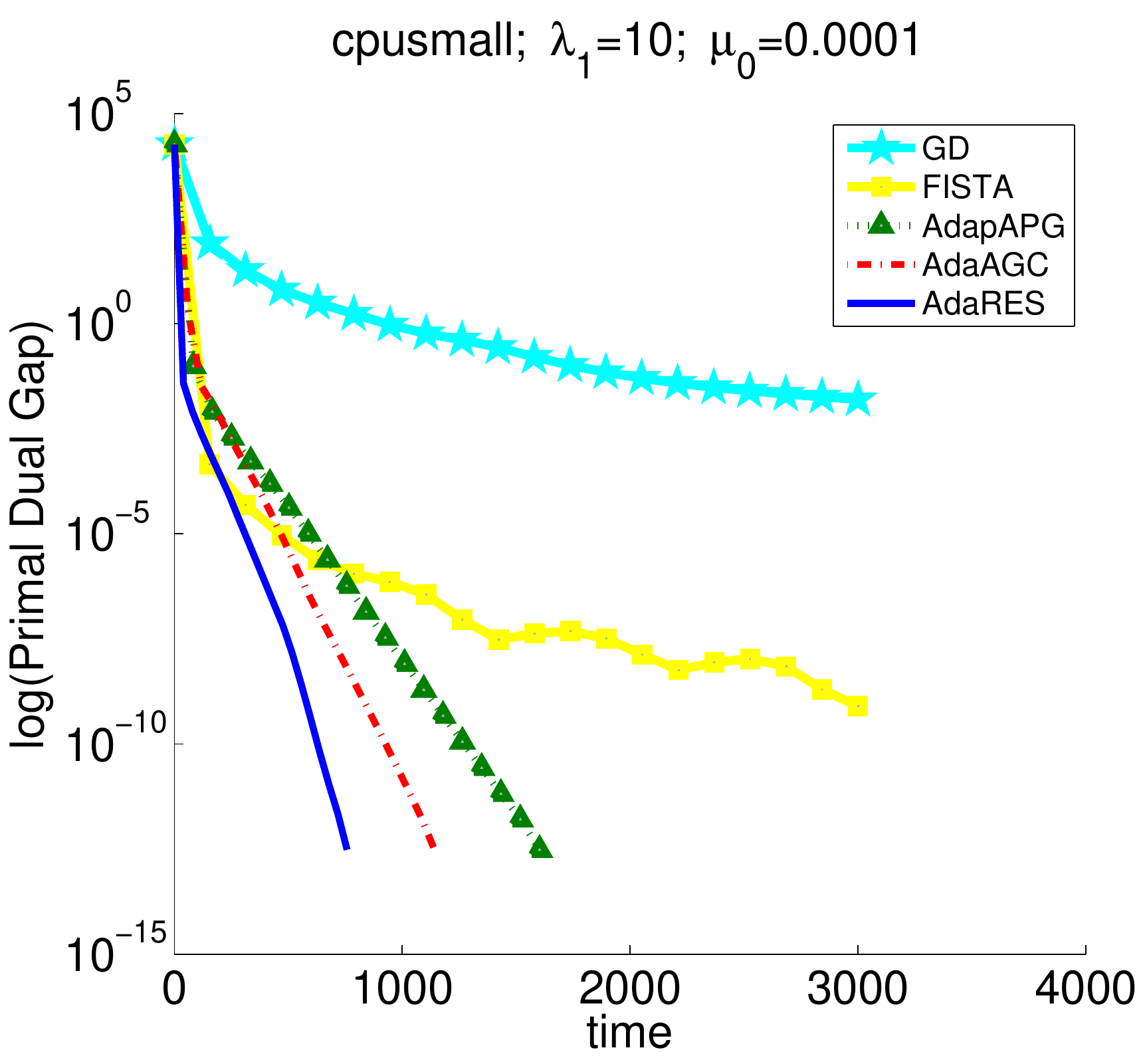}
\includegraphics[width=\sizefigure]{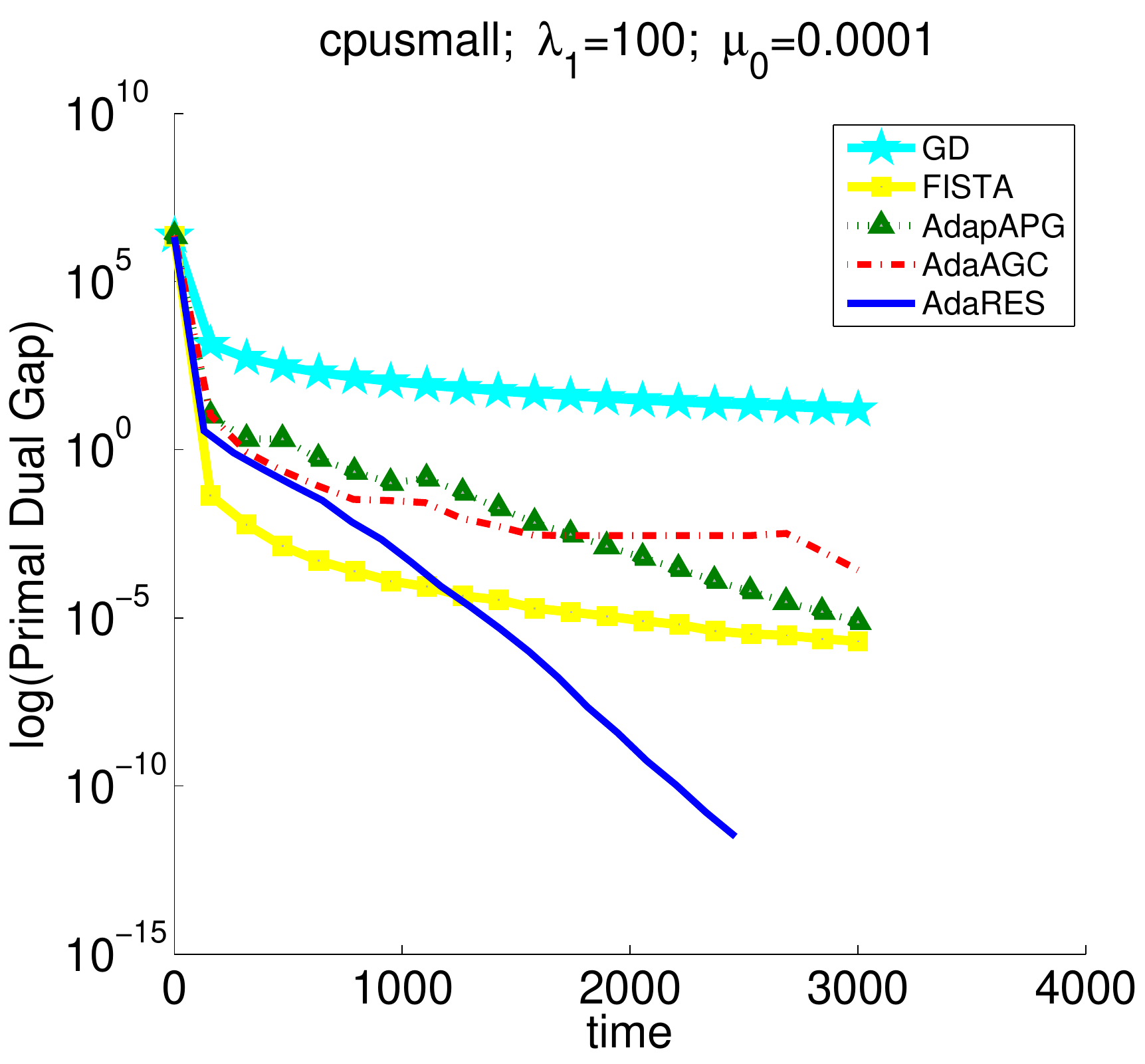}
\includegraphics[width=\sizefigure]{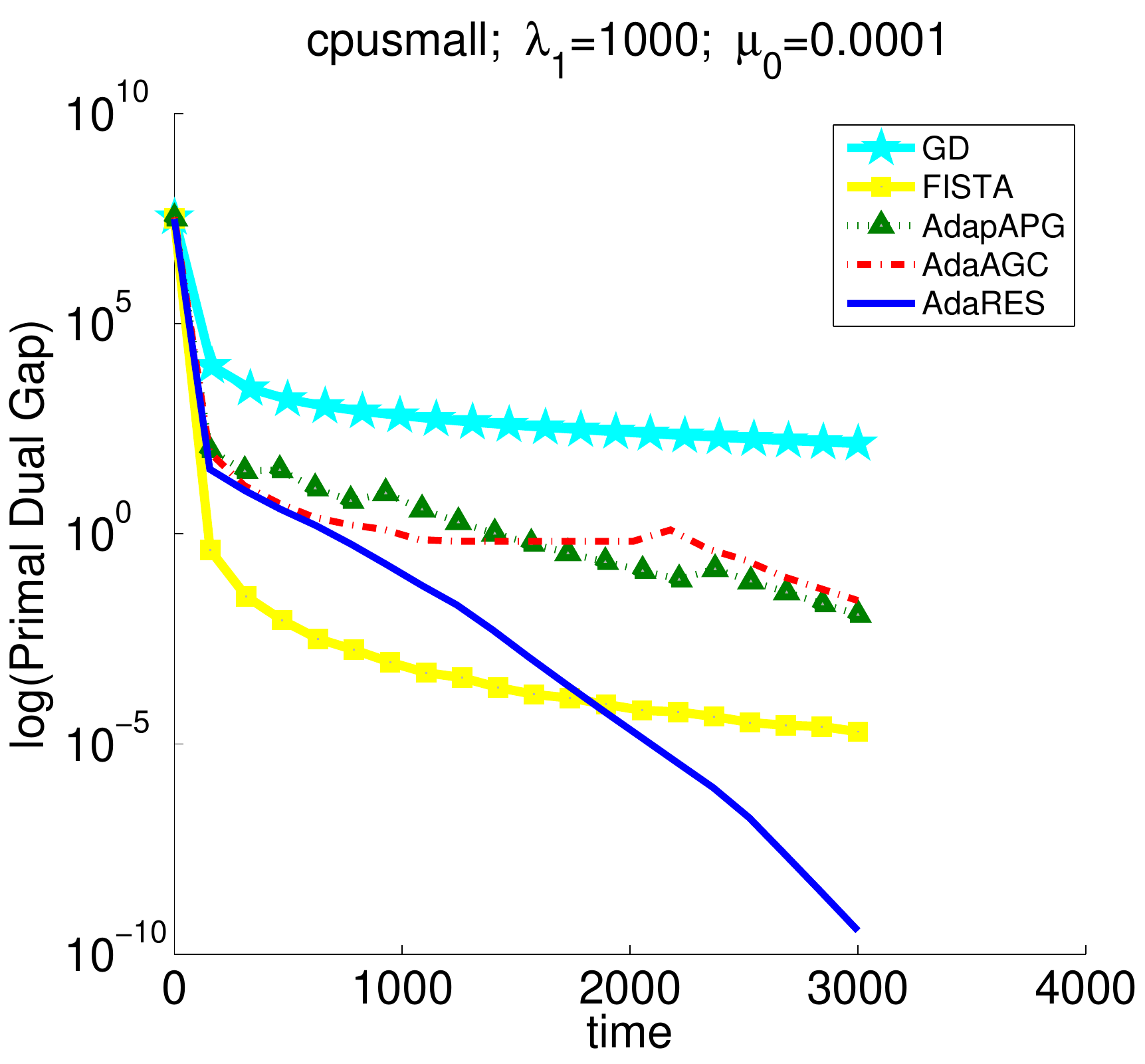}
\includegraphics[width=\sizefigure]{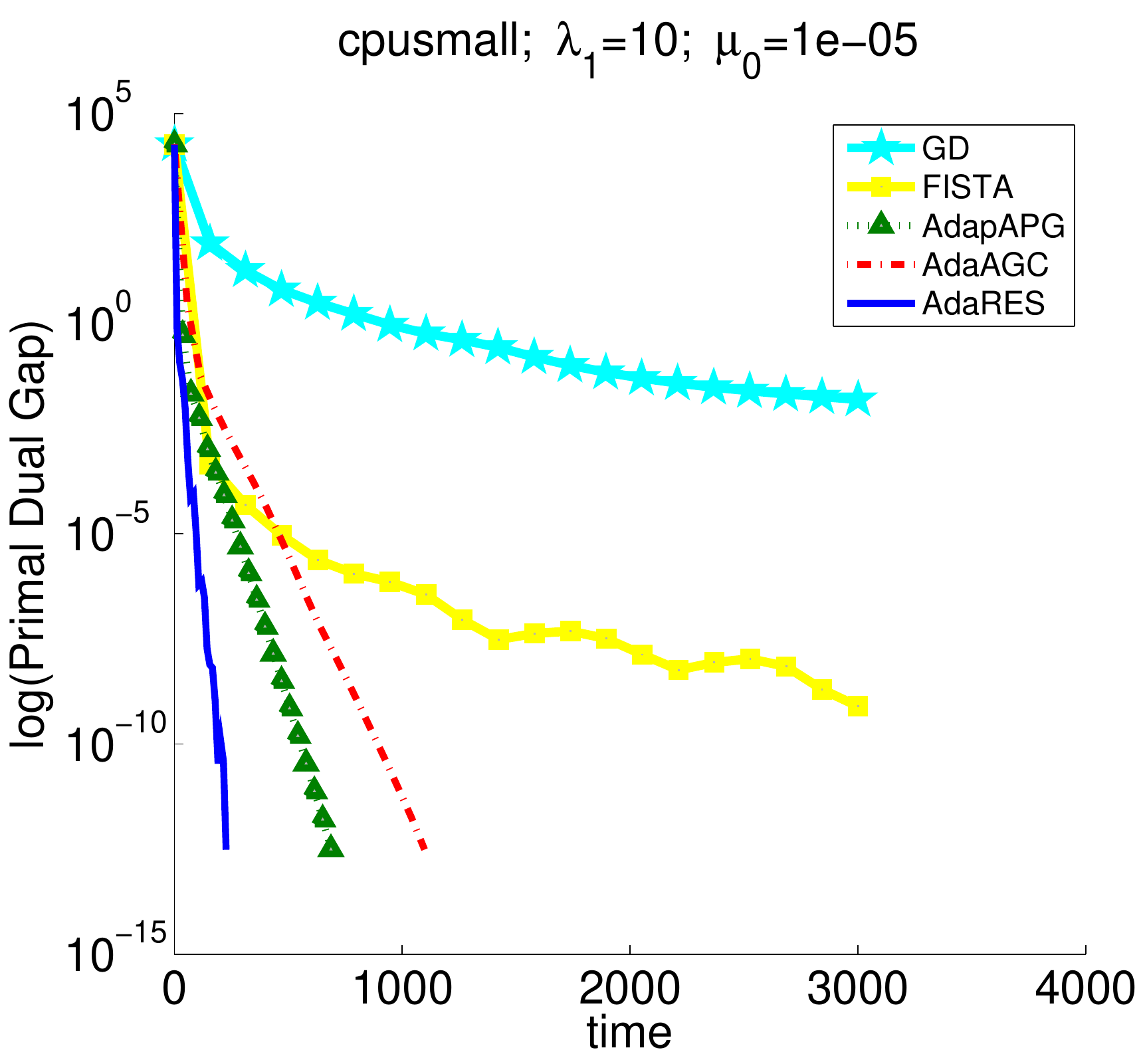}
\includegraphics[width=\sizefigure]{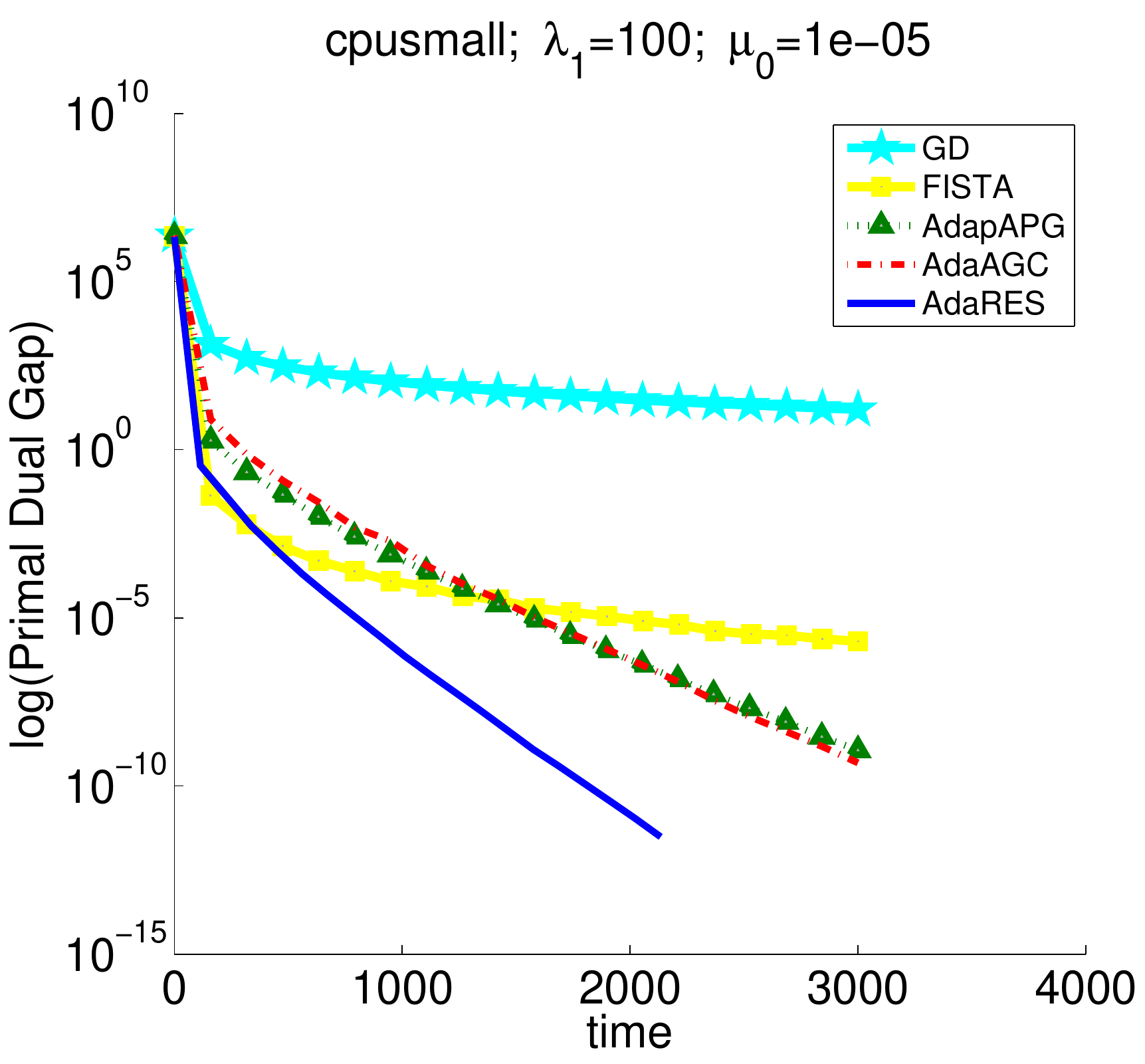}
\includegraphics[width=\sizefigure]{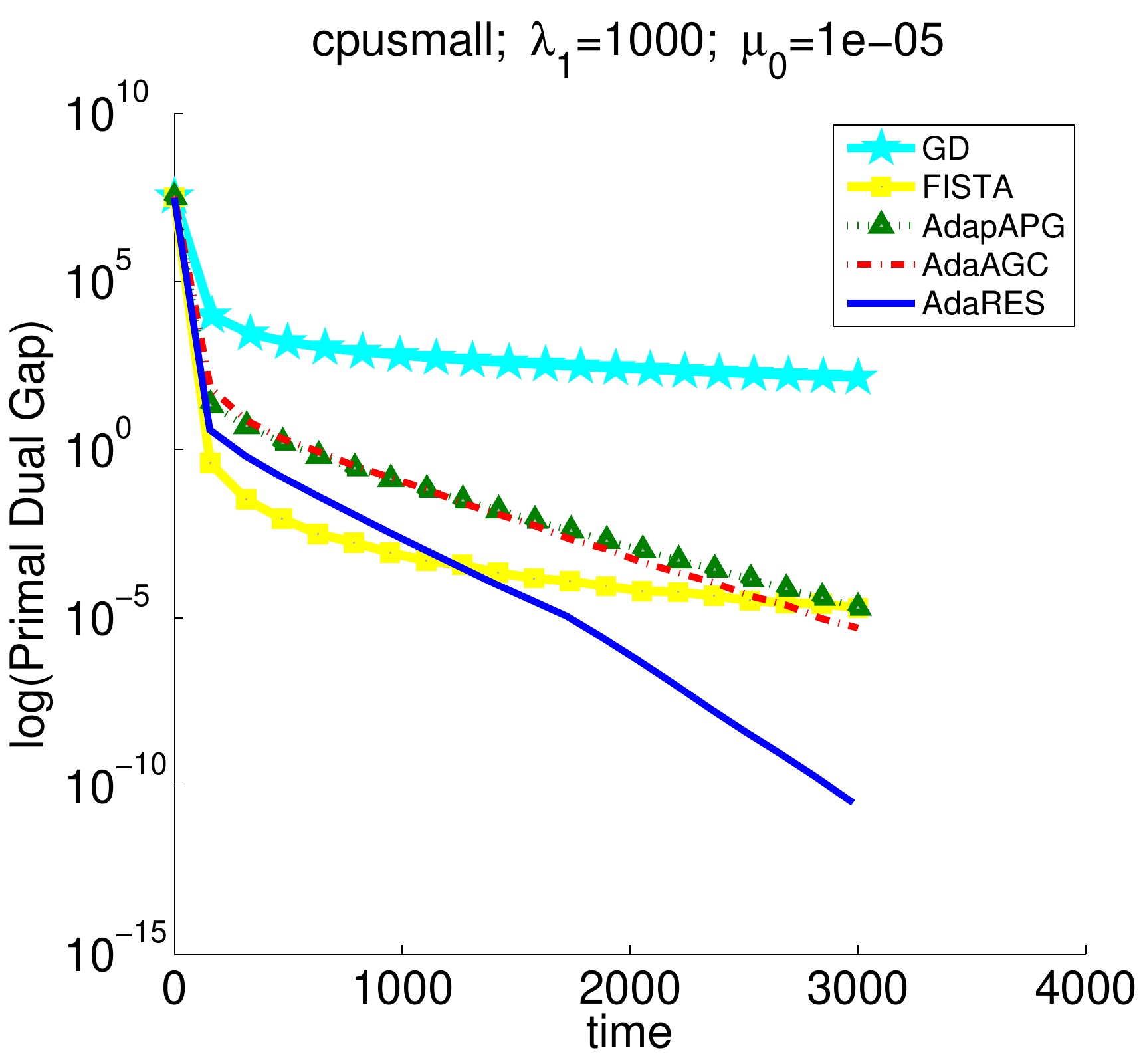}
\includegraphics[width=\sizefigure]{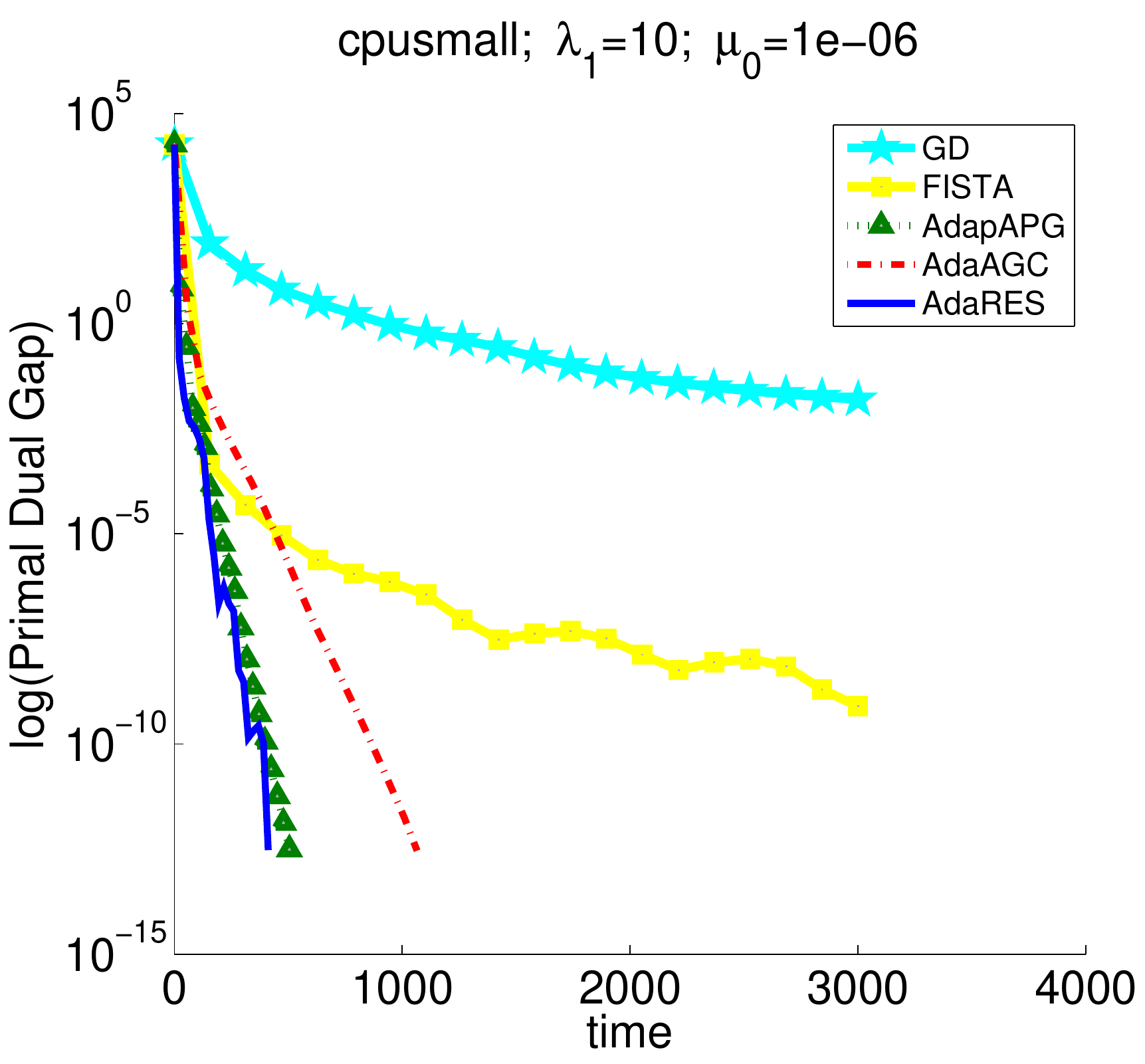}
\includegraphics[width=\sizefigure]{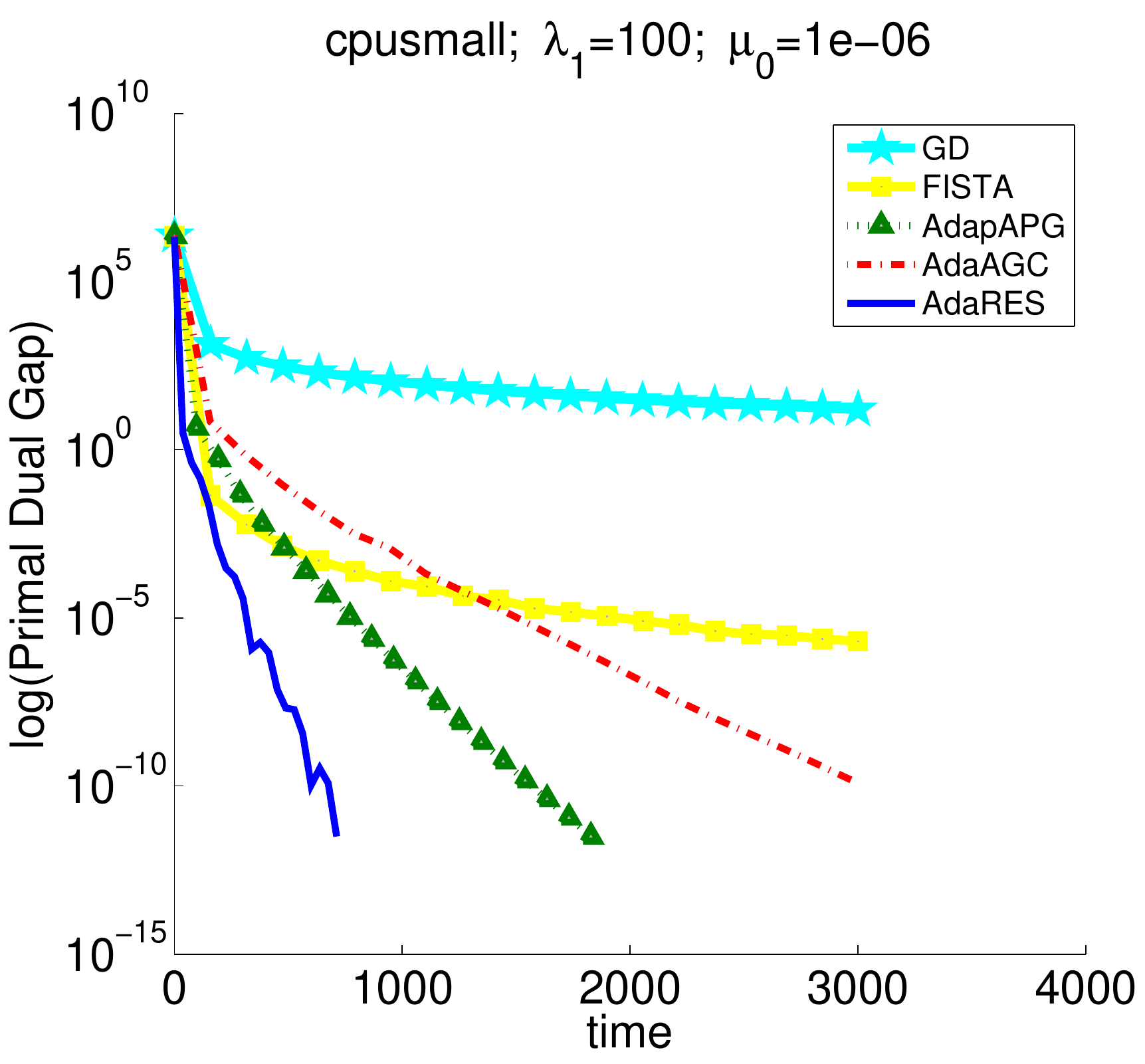}
\includegraphics[width=\sizefigure]{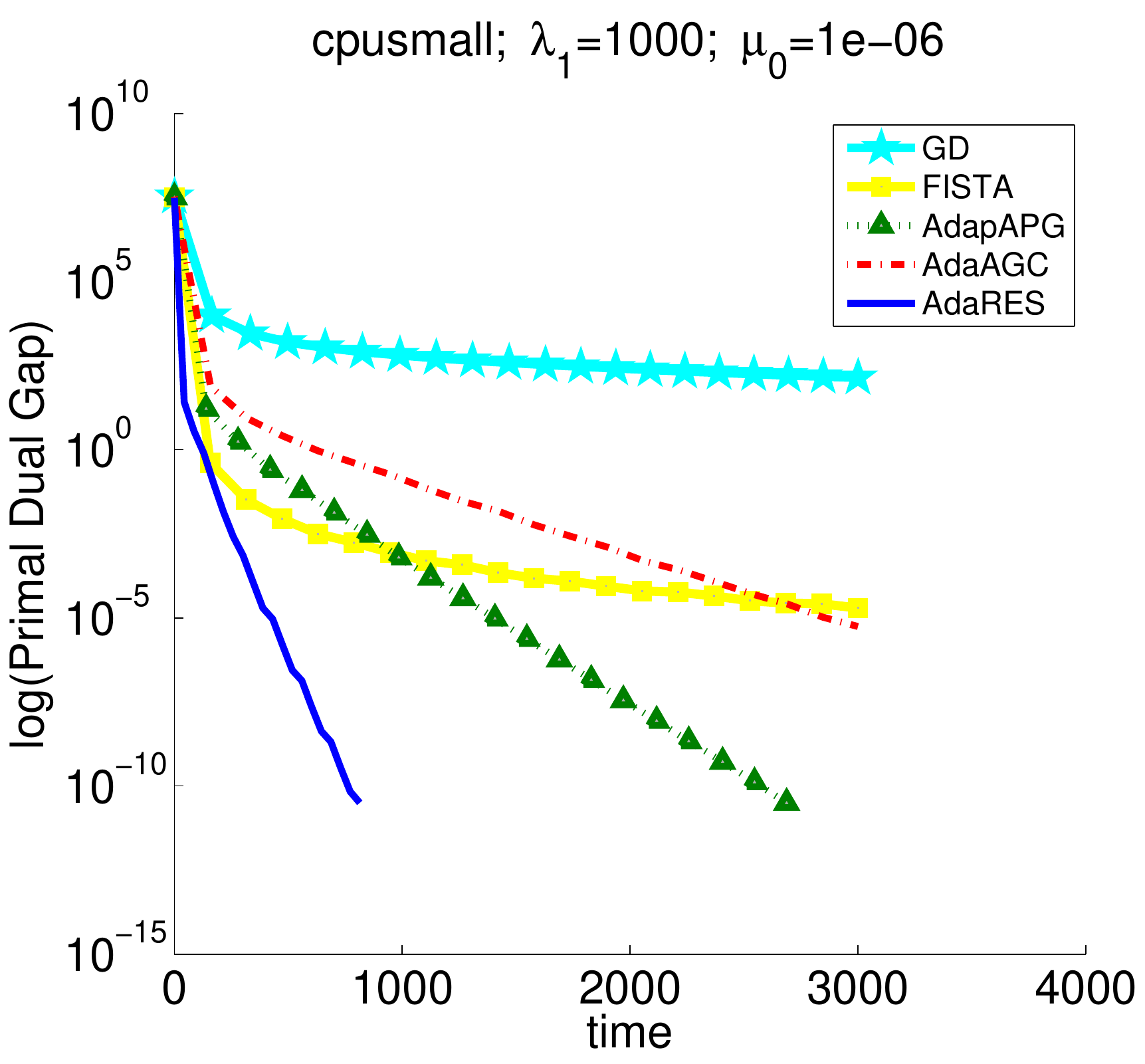}
\caption{Experimental results on the logistic regression problem~\eqref{alr} and the dataset dorothea. Column-wise: we solve the same problem with a different 
a priori on the quadratic error bound. Row-wise: we use the same a priori
on the quadratic error bound but the weight of the regularization is varying.
}
\label{fig:dorothea}
\end{figure}

\clearpage
}

\section{Conclusion}

In this work, we show that global linear convergence is guaranteed if we restart  at any frequency accelerated gradient methods  under a local quadratic growth condition. We then propose an adaptive restarting strategy  based on the decrease of the norm of proximal gradient mapping. Compared with similar methods dealing with unknown local error bound condition number, our algorithm has a better worst-case complexity bound and practical performance. 

Our algorithm can be further extended to a more general setting when H\"olderian error bound~\eqref{a:holderian} is satisfied. 
Another avenue of research is that the accelerated coordinate descent method~\cite{FR:2013approx} 
faces the same issue as full gradient
methods: to get an accelerated rate of convergence, on needs to estimate the
strong convexity coefficient~\cite{lin2014accelerated}.
In~\cite{fercoqqu2016restarting}, an algorithm with fixed periodic restart was proposed.
We may also consider adaptive restart for
the accelerated coordinate descent method, to get more efficiency in large-scale computation.

\appendix

\section{proof of Lemma~\ref{l:thetak}, Lemma~\ref{l:xkzk}, Proposition~\ref{prop:nonblowout} and Proposition~\ref{prop:fista_basic}}
\begin{proof}[proof of Lemma~\ref{l:thetak}]
The equation~
\eqref{arectheta} holds because $\theta_{k+1}$ is the unique positive square root
to the polynomial $P(X) = X^2 + \theta_k^2 X - \theta_k^2$. \eqref{atheradecr} is a direct
consequence of \eqref{arectheta}.

Let us prove \eqref{athetabd} by induction. 
It is clear that $\theta_0 \leq \frac{2}{0+2}$. 
Assume that $\theta_k \leq \frac{2}{k+2}$.
We know that $P(\theta_{k+1}) = 0$ and that $P$ is an increasing function on $[0 , +\infty]$. So we just need to show that $P\big(\frac{2}{k+1+2}\big)\geq 0$.
\begin{align*}
P\Big(\frac{2}{k+1+2}\Big) = \frac{4}{(k+1+2)^2} 
+ \frac{2}{k+1+2}\theta_k^2 - \theta_k^2
\end{align*}
As $\theta_k \leq \frac{2}{k+2}$ and $\frac{2}{k+1+2}-1 \leq 0$,
\begin{align*}
P\Big(\frac{2}{k+1+2}\Big) &\geq \frac{4}{(k+1+2)^2} 
+ \Big(\frac{2}{k+1+2} - 1\Big) \frac{4}{(k+2)^2} \\
 &= \frac{4}{(k+1+2)^2 (k+2)^2} \geq 0.
\end{align*}

For the other inequality,  $\frac{1}{0+1} \leq \theta_0$.
We now assume that $\theta_k \geq \frac{1}{k+1}$ but
that $\theta_{k+1} < \frac{1}{k+1+1}$. 
Remark that $(x \mapsto (1-x)/x^2)$ is strictly decreasing for $x \in (0,2)$. Then, using \eqref{arectheta}, 
we have
\begin{align*} 
(k+1+1)^2 -(k+1+1) <\frac{1-\theta_{k+1}}{\theta_{k+1}^2}  \overset{\eqref{arectheta}}{=} \frac{1}{\theta_k^2} \leq (k+1)^2.
\end{align*}
This is equivalent to
\begin{equation*}
(2 - 1) (k+1) + 1 < 1
\end{equation*}
which obviously does not hold for any $k \geq 0$.
So $\theta_{k+1} \geq \frac{1}{k+1+1}$.
\end{proof}

\begin{proof}[proof of Lemma~\ref{l:xkzk}]
The relation~\eqref{a:xkzk} follows by combining line 3 and line 5   of Algorithm~\ref{FISTA} and Algorithm~\ref{APG}.
\end{proof}

\begin{proof}[proof of Proposition~\ref{prop:nonblowout}]
 FISTA can be written equivalently using only $(x_k, y_k)$ and 
$$
y_k=x_k+\beta_k (x_{k}-x_{k-1})
$$ 
where $\beta_k=\theta_k(\theta_{k-1}^{-1}-1)\leq 1$.
\begin{align*}
F(&x_{k+1})=f(y_k+(x_{k+1}-y_k))+\psi(x_{k+1})
\\&\leq f(y_k)+\<\nabla f(y_k), x_{k+1}-y_k>+\frac{1}{2}\|x_{k+1}-y_k\|_v^2+\psi(x_{k+1})
\\& =f(y_k)+\<\nabla f(y_k), x_k-y_k>+\<\nabla f(y_k), x_{k+1}-x_k>+\frac{1}{2}\|x_{k+1}-y_k\|_v^2+\psi(x_{k+1})
\\&\leq f(x_k)+\<\nabla f(y_k), x_{k+1}-x_k>+\frac{1}{2}\|x_{k+1}-y_k\|_v^2+\psi(x_{k+1})
\end{align*}
The update of FISTA implies the existence of $\xi \in \partial \psi(x_{k+1})$ such that
$$
\nabla f(y_k)+v\cdot (x_{k+1}-y_k)+\xi=0.
$$
Therefore,
\begin{align*}F(x_{k+1}) 
&\leq f(x_k)+\<\xi+v\cdot (x_{k+1}-y_k), x_k-x_{k+1}>+\frac{1}{2}\|x_{k+1}-y_k\|_v^2+\psi(x_{k+1})
\\&\leq F(x_k) +\<v\cdot (x_{k+1}-y_k), x_k-x_{k+1}>+\frac{1}{2}\|x_{k+1}-y_k\|_v^2
\\&=F(x_k)+\frac{1}{2}\|x_k-y_k\|_v^2-\frac{1}{2}\|x_k-x_{k+1}\|_v^2
\\&=F(x_k)+\frac{\beta_k^2}{2}\|x_k-x_{k-1}\|_v^2-\frac{1}{2}\|x_k-x_{k+1}\|_v^2
\\&\leq F(x_k)+\frac{1}{2}\|x_k-x_{k-1}\|_v^2-\frac{1}{2}\|x_k-x_{k+1}\|_v^2.
\end{align*}
By applying the last inequality recursively we get $F(x_k)\leq F(x_0)$.

Next consider the iterates of APG.
By~\cite[Proposition 1]{tseng2008accelerated}, for any $x\in \R^n$,
the iterates of APG satisfy the following property:
$$
F(x_{k+1})-F(x)+\frac{\theta_k^2}{2}\|x-z_{k+1}\|^2_v
\leq (1-\theta_{k})\left(F(x_{k})-F(x)\right)+\frac{\theta_k^2}{2}\|x-z_{k}\|^2_v.
$$
By taking $x=x_k$ we obtain:
$$
F(x_{k+1})-F(x_k)
\leq \frac{\theta_k^2}{2}\|x_k-z_{k}\|^2_v-\frac{\theta_k^2}{2}\|x_k-z_{k+1}\|^2_v
$$
The update of APG implies:
$$
x_{k+1}=(1-\theta_k)x_k+\theta_k z_{k+1},
$$
which yields 
$$
x_{k+1}-z_{k+1}=(1-\theta_k)(x_k-z_{k+1}).
$$
Therefore,  for any $k\geq 1$,
\begin{align}\label{a:FISTAnonblowout}
F(x_{k+1})-F(x_k)
\leq \frac{\theta_k^2}{2}\|x_k-z_{k}\|^2_v-\frac{\theta_k^2}{2(1-\theta_k)^2}\|x_{k+1}-z_{k+1}\|^2_v
\end{align}
Applying~\eqref{a:FISTAnonblowout} recursively and using the decreasing property~\eqref{atheradecr} we obtain:
$$
F(x_{k+1})-F(x_0)\leq -\frac{1}{2}\|x_0-z_1\|_v^2\leq 0.
$$
\end{proof}

\begin{proof}[proof of Proposition~\ref{prop:fista_basic}]
We first prove~\eqref{eq:itcompl_fista}  for FISTA. Let any $x_\star\in\cX_\star$.
Since $z_{k+1} = z_{k} + \theta_{k}^{-1}(x_{k+1} - y_{k}) = \theta_{k}^{-1} x_{k+1} - (\theta_{k}^{-1} - 1) x_{k}$, Inequality \eqref{eq:itcompl_fista} is a simple consequence of Lemma~4.1 in~\cite{beck2009fista} (Note that we have a shift of indices for our variables $(x_{k+1}, z_{k+1})$ vs $(x_k, u_k)$ in~\cite{beck2009fista}). 

To prove~\eqref{eq:itcompl_fista}  for APG, we notice that it is a special case of  Theorem~3 in~\cite{FR:2013approx}. 
\end{proof}

\section*{Acknowledgement}

The first author's work was supported by the
EPSRC  Grant EP/K02325X/1
{\em Accelerated Coordinate Descent Methods for Big Data Optimization},
the Centre for Numerical Algorithms and Intelligent Software (funded by EPSRC grant EP/G036136/1 and the Scottish Funding Council), the Orange/Telecom ParisTech think tank Phi-TAB and the ANR grant ANR-11-LABX-0056-LMH, LabEx LMH as part of the Investissement d'avenir project.  The second author's work was
supported  by Hong Kong Research Grants Council Early Career Scheme 27303016.
This research was conducted using the HKU Information Technology Services research computing facilities that are supported in part by the Hong Kong UGC Special Equipment Grant (SEG HKU09).

\bibliographystyle{alpha}
\bibliography{../literature}

\end{document}